\documentclass[12pt,fleqn]{article}

\usepackage{amsmath,amssymb,amsthm}
\usepackage{geometry} 
\usepackage[pdfpagemode=UseNone,pdfstartview=FitH]{hyperref}
\usepackage{enumitem}
\usepackage{comment}
\usepackage{xcolor}
\usepackage{esint}
\usepackage{graphicx}

\newcommand{\A}{{\cal A}}

\newcommand{\F}{{\cal F}}

\newcommand{\dist}[2]{\text{dist}\, (#1,#2)}

\newcommand{\M}{{\cal M}}
\newcommand{\N}{\mathbb N}

\newcommand{\pnorm}[2][]{\if #1'' \left|#2\right|_p \else \left|#2\right|_{#1} \fi}
\newcommand{\R}{\mathbb R}

\newcommand{\seq}[1]{\left(#1\right)}
\newcommand{\set}[1]{\left\{#1\right\}}

\newcommand{\w}[1]{\widetilde{#1}}
\newcommand{\weak}{\rightharpoonup}

\newcommand{\J}{\mathcal J}

\newcommand{\om}{\int_{\mathbb{R}^N}}

\newenvironment{enumroman}{\begin{enumerate}

}{\end{enumerate}}

\newtheorem{corollary}{Corollary}[section]
\newtheorem{lemma}[corollary]{Lemma}
\newtheorem{proposition}[corollary]{Proposition}
\newtheorem{theorem}[corollary]{Theorem}

\theoremstyle{definition}
\newtheorem{definition}[corollary]{Definition}

\theoremstyle{remark}

\newtheorem{remark}[corollary]{Remark}

\numberwithin{equation}{section}

\usepackage{authblk}

\title{\bfseries
Doubly critical mixed-order quasilinear equations: existence, multiplicity and regularity
}

\author[1]{Elisandra Gloss}
\author[2]{Artur Jorge Marinho}
\author[3]{Kanishka Perera}
\author[4]{Bruno Ribeiro}

\vspace{0.5cm}

\affil[1,4]{Departamento de Matem\'atica\\
Universidade Federal da Para\'iba\\

Cidade Universit\'aria, 58051-900, Jo\~ao Pessoa, Brazil\\
\texttt{elisandra.gloss@academico.ufpb.br \& bhcr@academico.ufpb.br}}

\affil[2,3]{Department of Mathematics\\
Florida Institute of Technology\\

150 W University Blvd, Melbourne, FL 32901-6975, USA\\
\texttt{amarinho2024@my.fit.edu \& kperera@fit.edu}}

\date{}

\begin{document}

\maketitle

\vspace{0.5cm}

\begin{abstract}
We investigate a class of mixed-order quasilinear elliptic equations in
\(\mathbb R^N\) driven by the \(q\)-biharmonic and \(p\)-Laplacian
operators,
\[
\Delta_q^2u-\Delta_pu
=
\lambda\frac{|u|^{r_\sigma-2}u}{|x|^\sigma}
+|u|^{p^*-2}u
+|u|^{q^{**}-2}u,
\]
where \(1<q<N/2\), \(1<p<q^*\), \(p\ne q\), \(q\ne p^*\), and the
weighted exponent is determined by the natural scaling of the equation.
We establish compact weighted embeddings compatible with this mixed
scaling and develop the corresponding variational framework. As an
application, we prove existence and multiplicity of nontrivial solutions
under suitable assumptions on the parameters. We also establish a local
regularity result for the more general equation $\Delta_q^2u-\Delta_pu=f(x,u),$ where \(f\) is a Carath\'eodory function with local critical growth. By
combining nonlinear potential estimates, a regularity-lifting argument,
and a finite bootstrap for an associated second-order system, we obtain
higher local regularity of both \(u\) and the nonlinear flux
\(|\Delta u|^{q-2}\Delta u\). This regularity result, which is also of
independent interest, enables us to derive a Pohozaev-type identity for
the equations under consideration.
\end{abstract}

\newpage

\tableofcontents

\section{Introduction}

In this paper we study mixed quasilinear elliptic problems on $\R^N$ involving the
$q$-biharmonic operator and the $p$-Laplacian. More precisely, we consider equations of the
form
\begin{equation}\label{intro:prob1}
    \Delta_q^2u-\Delta_pu
    =
    \lambda\dfrac{|u|^{r_\sigma-2}u}{|x|^{\sigma}}
    + \tau_1|u|^{p^*-2}u
    + \tau_2|u|^{q^{**}-2}u
    \quad\text{in }\R^N,
\end{equation}
where $\Delta_q^2u=\Delta\big(|\Delta u|^{q-2}\Delta u\big)$ is the $q$-biharmonic
operator, $\Delta_pu =\mathrm{div}\big(|\nabla u|^{p-2}\nabla u\big)$ is the $p$-Laplacian
operator, $\lambda\in\R$, $\tau_1,\tau_2\in\{0,1\}$ and
\begin{equation}\label{p,q,q*}
1<q<\frac N2,
\qquad
1<p<q^*,\qquad p\neq q,\qquad q\ne p^*.
\end{equation}
with $\sigma\in[0,N)$ {and} $1<r_\sigma<2q.$  The precise restrictions on $\sigma$ and $r_\sigma$ will be given later. As usual,
$p^*={Np}/({N-p})$ and $q^{**}={Nq}/({N-2q})$.
Notice that $q<N/2$ is equivalent to $q^*<N$ and also $q^*<2q$. So, we also have $p<N$ and $p<2q$. 

\medskip

Quasilinear fourth-order equations driven by the $q$-biharmonic
operator have been studied in connection with nonlinear elasticity,
non-Newtonian fluids, and models involving higher-order diffusion.
Early variational results for Navier problems involving the
$q$-biharmonic operator were obtained in
\cite{CanditoMolicaBisci2012}, while critical-growth problems were
investigated, among others, in \cite{dos2010positive}. Ground-state
solutions for $q$-biharmonic equations in the whole space were
considered in \cite{LiuChenAlmuaalemi2017}. The presence of singular
Hardy--Rellich potentials introduces an additional loss of compactness
and has motivated a substantial literature on positive, ground-state,
and sign-changing solutions; see, for example,
\cite{DhifliAlsaedi2021,HuangLiu2014,YangZhangLiu2017} and the
references therein. These problems are naturally connected with the
critical embedding $D ^{2,q}(\mathbb R^N)\hookrightarrow
L^{q^{**}}(\mathbb R^N)$ and with the corresponding Hardy--Rellich inequalities.

\medskip

A related line of research concerns fourth-order equations perturbed by a nonlinear second-order operator. Sun, Chu, and Wu \cite{SunChuWu2017} studied equations of the form
\[
\Delta^2u-\beta\Delta_pu+\lambda V(x)u=f(x,u)
\qquad\text{in }\mathbb R^N,
\]
establishing existence and multiplicity results  and Gagliardo--Nirenberg inequalities. Further results for singular
potentials and Nehari-type constraints were obtained in \cite{SunWu2019}, while sign-changing solutions under Neumann boundary conditions were considered in \cite{SunLiuWu2017}. In the critical
setting, Su and Shi \cite{SuShi2021} studied a biharmonic equation in $\mathbb R^N$ involving a critical $p$-Laplacian term and a Hardy--Rellich potential. These works already exhibit some of the difficulties caused by the interaction of operators of different
orders, but the leading fourth-order operator in them is the classical biharmonic operator.

Problems in which both the fourth-order and the second-order parts are
quasilinear are much less understood. A particularly relevant
predecessor is the work of Yu, Zhao, and Luo
\cite{YuZhaoLuo2022}, who proved the existence of a ground-state
solution for
\[
\Delta_q^2u-\Delta_pu
-\mu\frac{|u|^{q-2}u}{|x|^{2q}}
=|u|^{q^{**}-2}u
\qquad\text{in }\mathbb R^N
\]
in the endpoint case $p=q^*$. In this case the $q$-biharmonic energy, the $p$-Dirichlet energy,
the Hardy--Rellich term, and the $q^{**}$-critical nonlinearity are all
invariant under the same dilation. More recently, Yang et al.
\cite{YangCarlosHamdaniKefi2026} considered a genuinely mixed
$q$-biharmonic and $p$-Laplacian problem with logarithmic and
$q^{**}$-critical nonlinearities on a bounded domain. In contrast, the
present paper deals with the strictly subcritical relation $p<q^*$ in
the whole space and incorporates simultaneously the critical
nonlinearities associated with
$D^{1,p}(\mathbb R^N)$ and
$D^{2,q}(\mathbb R^N)$, together with a weighted Hardy--Sobolev perturbation. 

\medskip

One of the main difficulties in the analysis of \eqref{intro:prob1} is due to the fact that
the two principal operators possess different homogeneities. Our approach is based on the
identification of a suitable scaling under which these homogeneities become compatible.
This scaling places the problem within the abstract framework of scaled operators
developed in \cite{MR5043800}, leading naturally to an associated nonlinear eigenvalue problem. A detailed scaling analysis, carried out in the next section, determines the admissible values of the parameters $r$ and $\sigma$ and provides the compact weighted embeddings required by our variational approach.

More precisely, throughout this paper we work under the following assumptions
\begin{equation}\label{sigma_and_rsigma}
p<\sigma<2q\quad\text{and}\quad r_\sigma:=
\frac{pq+\sigma(q-p)}{2q-p}.
\end{equation}

The corresponding eigenvalue problem is
\begin{equation}\label{intro:eigenvalue-problem}
\Delta_q^2u-\Delta_pu
=
\lambda
\frac{|u|^{r_\sigma-2}u}{|x|^\sigma}
\quad\text{in }\R^N.
\end{equation}

That means, under a suitable choice of $\gamma$ and a suitable scaling of the form $u_t(x)=t^\gamma u(tx)$, this problem becomes invariant for the family $(u_t),$ $t>0$. Using the abstract theory of scaled operators and the Fadell--Rabinowitz cohomological
index, we obtain a sequence of nonlinear eigenvalues
\[
0<\lambda_1\le\lambda_2\le\cdots\nearrow+\infty,
\]
which plays the role of variational thresholds in our existence results.

\medskip

This common-scaling approach and the resulting nonlinear spectral
theory are not restricted to the mixed-order problem considered here. They have proved successful in the study of nonlinear PDEs whose constituent terms acquire a common nontrivial homogeneity after a suitable combined spatial and amplitude dilation. The abstract framework developed in \cite{MR5043800} was applied to fractional
Schrödinger--Poisson--Slater equations in Coulomb--Sobolev spaces in
\cite{MR5047855}. Further applications include inhomogeneous Schrödinger equations with competing singular nonlinearities \cite{GlossPereraRibeiro2026}, critical equations with superpositions of nonlocal Hartree-type nonlinearities \cite{MarinhoPerera2026}, and new existence and multiplicity results for Schrödinger--Poisson--Slater equations in Coulomb--Sobolev spaces \cite{MarinhoMercuriPerera2026}.

\medskip

Our first goal is to establish the existence and multiplicity of nontrivial solutions for critical
perturbations of \eqref{intro:eigenvalue-problem}, that is, \eqref{intro:prob1} with $\tau_1=1$ or/and $\tau_2=1$. Depending on the relative position of
the critical exponents $p^*$ and $q^{**}$ with respect to the scaling, the problem naturally
splits into different parameter regimes. In some of these regimes both critical
nonlinearities may be treated simultaneously, while in others one of them must be removed
in order to recover the compactness needed by the variational argument.

Let us define a weak solution for equations like \eqref{intro:prob1}. Given a Carath\'eodory function $f:\mathbb{R}^N\times\mathbb{R}\to\mathbb{R}$, we say that $u\in D^{1,p}(\mathbb{R}^N)\cap
D^{2,q}(\mathbb R^N)$ is a weak solution for
\begin{equation}\label{generalfproblem}
\Delta_q^2u-\Delta_pu=f(x,u)\quad\text{in  } \R^N
\end{equation}
if $f(\cdot,u)\in L^1_{loc}(\R^N)$ and
\[
\int_{\R^N}|\Delta u|^{q-2}\Delta u \Delta vdx+\int_{\R^N}|\nabla u|^{p-2}\nabla u \nabla vdx=\int_{\R^N} f(x,u)vdx,\quad\forall v\in C^\infty_c(\R^N).
\]
With this notion of solution in mind, let us establish our main results.

\begin{theorem}\label{intro:thm:p-less-q}
Assume that $1<q<N/2$ and
\(
1<p<q<p^*\)
or
\(
1<q<p<q^*.
\)
Consider 
\begin{equation}\label{prob1}
    \Delta_q^2u-\Delta_pu
    =
    \lambda\dfrac{|u|^{r_\sigma-2}u}{|x|^{\sigma}}
    + |u|^{p^*-2}u
    + |u|^{q^{**}-2}u
    \quad\text{in }\R^N,
\end{equation}
for $\sigma$ and $r_\sigma$ as in \eqref{sigma_and_rsigma}. Suppose that
\[
\lambda_k=\cdots=\lambda_{k+m-1}<\lambda_{k+m}
\]
for some $k,m\geq1$. Then there exists $\delta_k>0$ such that, for every
\(
\lambda\in(\lambda_k-\delta_k,\lambda_k),
\)
problem \eqref{prob1} possesses $m$ distinct pairs of nontrivial weak solutions at positive energy levels.
\end{theorem}

When $q>p^*$, the previous theorem cannot be expected to hold for the full
problem \eqref{prob1}. In these situations, the natural scaling suggests that
one of the critical nonlinearities should be removed. This leads to the following
complementary results.

\begin{theorem}\label{intro:thm:p-less-q-large}
Assume that $1<p<p^*<q<N/2$ and consider the equation
\begin{equation}\label{intro:problem-p-less-q-large}
\Delta_q^2u-\Delta_pu
=
\lambda\dfrac{|u|^{r_\sigma-2}u}{|x|^\sigma}
+
|u|^{q^{**}-2}u
\quad\text{in }\R^N,
\end{equation}
where $\sigma$ and $r_\sigma$ satisfy \eqref{sigma_and_rsigma}. Suppose that
\[
\lambda_k=\cdots=\lambda_{k+m-1}<\lambda_{k+m}
\]
for some $k,m\geq1$. Then there exists $\delta_k>0$ such that, for every
\(
\lambda\in(\lambda_k-\delta_k,\lambda_k),
\)
problem \eqref{intro:problem-p-less-q-large} possesses $m$ distinct pairs of nontrivial weak solutions at positive energy levels.
\end{theorem}

The proofs rely on a scaling-based linking construction together with the
Palais--Smale compactness condition established below an explicit threshold $c^*>0$ by
means of concentration -- compactness arguments. The interaction between the two critical
Sobolev nonlinearities determines this threshold in the full problem, while in the reduced
problems it simplifies to the corresponding one-critical obstruction.

\medskip

A further contribution of the paper is a local regularity result for a
substantially more general class of mixed-order equations. This result is
independent of the variational and scaling structures used in the existence
theory: for a weak solution $u$ of \eqref{generalfproblem}, we obtain higher local regularity on any open set
$\Omega \subset \mathbb{R}^N$, requiring the nonlinearity $f$ to satisfy only
a local critical-growth condition on $\Omega$. Besides being of independent interest, this result provides the annular regularity needed to justify the Pohozaev identities
used at several stages of the existence and compactness arguments. More
precisely, we prove the following result.

\medskip

Let $f:\mathbb R^N\times\mathbb R\to\mathbb R$ be a Carath\'eodory
function. Suppose that $\Omega\subset\mathbb R^N$ is an open set, and assume that $f$ satisfies the following local critical-growth condition on $\Omega$: for
every bounded domain $\Omega'\Subset\Omega$, there exists
$C_{\Omega'}>0$ such that
\begin{equation}\tag{$f_0$}\label{condf_0}
|f(x,t)|\leq C_{\Omega'}\bigl(1+|t|^{q^{**}-1}\bigr)
\end{equation}
for a.e. $x\in\Omega'$ and every $t\in\mathbb R$. 

\begin{theorem}\label{thm:intro-regularity}
Suppose $f(x,t)$ satisfies \eqref{condf_0}. Assume that $1<q< N/2,$  $1<p<q^*$. Then every weak solution of 
\eqref{generalfproblem} satisfies
\begin{equation}\label{all estimate}
u\in W_{\mathrm{loc}}^{2,m}
(\Omega),
\qquad
w:=|\Delta u|^{q-2}\Delta u
\in
W_{\mathrm{loc}}^{1,m}
(\Omega),\quad\forall\ 1\le m<\infty.
\end{equation}
Consequently, $u\in C_{\mathrm{loc}}^{1,\alpha} (\Omega),$
 and  $w\in C_{\mathrm{loc}}^{0,\alpha} (\Omega),$ for every \(\alpha\in(0,1)\). 
\end{theorem}

\begin{remark}
    We formulate the conclusion in terms of the pair $(u,w)$, since the
relation $\Delta u=|w|^{q'-2}w$ does not, when $q>2$, generally yield
$u\in W_{\mathrm{loc}}^{3,m}$ due to the singularity of the inverse
power map at the zeros of $w$.
\end{remark}

Theorem~\ref{thm:intro-regularity} is proved by a regularity mechanism that
differs from the truncation arguments usually employed in critical
quasilinear problems. We first rewrite the fourth-order equation as a
coupled second-order system for \(u\) and the nonlinear flux $w=|\Delta u|^{q-2}\Delta u.$ After localization, we derive a pointwise inequality whose
leading term has the structure of a nonlinear Havin--Maz'ya-type potential.
Such potentials, introduced in~\cite{MR0409858}, play a fundamental role in
nonlinear potential theory and in regularity estimates for quasilinear
equations; see, for example,
\cite{CianchiSchwarzacher2018,KuusiMingione2014}. In the present setting,
however, the measure entering the potential depends critically on the
unknown itself.

Combining a Brezis--Kato decomposition~\cite{MR0539217} of the critical
coefficient with the regularity-lifting philosophy for integral systems
developed in~\cite{ChenLi2010,MaChenLi2011}, we obtain a contraction that
produces an initial integrability gain beyond \(q^{**}\). When \(1<q<2\),
the resulting nonlinear potential is not subadditive, and this difficulty is
overcome by a nonlinear change of variable. The initial gain is then
propagated through a finite bootstrap for the coupled system, yielding all
finite local integrability exponents for both \(D^2u\) and \(\nabla w\).
Thus, the argument applies to general critical-growth nonlinearities and
provides an alternative to the usual truncation methods, while recovering
in particular the regularity required for the Pohozaev identities in this
paper.

\medskip

The remainder of the paper is organized as follows. In Section~\ref{sectionsobolev}, we establish the Sobolev and weighted embedding results needed throughout the paper. In Section~\ref{scaled theory}, we introduce the appropriate scaling, verify the hypotheses of the abstract framework of scaled operators, and construct the associated sequence of minimax eigenvalues. Section~\ref{section-conccomp} collects the concentration--compactness tools used to recover compactness in the presence of the critical nonlinearities. In Sections~\ref{sec-prof11} and~\ref{sec-proof12}, we prove Theorems~\ref{intro:thm:p-less-q} and~\ref{intro:thm:p-less-q-large}, respectively, treating the different ranges of \(p\) and \(q\). Section~\ref{sec-regularity} is devoted to the regularity theory for weak solutions and contains the proof of Theorem~\ref{thm:intro-regularity}. Finally, in Section~\ref{sec:pohozaev-identity}, we establish the Pohozaev identity and derive from it a standard nonexistence result.

\section{Sobolev embedding results}\label{sectionsobolev}  

For \(0< \eta < N\) and \(r\in[ 1,\infty)\), we denote by
\[
L_\eta^r(\R^N):=
\left\{
u:\R^N\to\R\ \text{measurable}:\;
\int_{\R^N}\frac{|u|^r}{|x|^\eta}\,dx<\infty
\right\}
\]
the weighted Lebesgue space endowed with its usual norm
\[
\|u\|_{L_\eta^r}:=
\left(
\int_{\R^N}\frac{|u|^r}{|x|^\eta}\,dx
\right)^{1/r}.
\]
When $\eta=0$ we have the  Lebesgue space $L^r(\R^N)$, whose norm 
we simply denote by $\|\cdot\|_r$.

We also consider the standard Banach spaces
\[
{D}^{1,p}(\R^N)=\{u\in L^{p^*}(\R^N):|\nabla u|\in L^{p}(\R^N)\}
\]
and 
\[
{D}^{2,q}(\R^N)=\{u\in L^{q^{**}}(\R^N):|D^2u|\in L^{q}(\R^N)\},
\]
where $D^2u=(u_{x_ix_j})_{i,j}$, with norms
\[
\|u\|_{{D}^{1,p}}=\|\nabla u\|_p\quad\text{and}\quad \|v\|_{{D}^{2,q}}=\|\Delta v\|_q.
\]
We define
\[
\mathcal V:={D}^{1,p}(\R^N)\cap L^{q^{**}}(\R^N)\quad\text{and}\quad \mathcal W:={D}^{1,p}(\R^N)\cap {D}^{2,q}(\R^N),
\]
endowed, respectively, with the norms
\[
\|u\|_{\mathcal V}:=\|\nabla u\|_{p}+\|u\|_{{q^{**}}}\quad\text{and}\quad \|u\|:=\|\nabla u\|_{p}+\|\Delta u\|_{{q}}.
\]
As usual, for $p\in(1,N)$, $q\in(1,N/2)$ and $\eta\in[0,N)$, we define
\[
p_\eta^*:=\frac{p(N-\eta)}{N-p},
\qquad
q_\eta^{**}:=\frac{q(N-\eta)}{N-2q},
\]
and, since $p< q^*$, it holds
\(
p_\eta^*<q_\eta^{**}.
\)

\begin{lemma}
Consider $p\in(1,N)$ and $q\in(1,N/2)$. Then:
\begin{enumroman}
\item
For every \(0\le\eta\le p\) it holds the continuous embedding
\(
D^{1,p}(\R^N)\hookrightarrow L^{p_\eta^*}_\eta(\R^N).
\)
\item
For every \(0\le\eta\le2q\) we have the continuous embedding
\(
D^{2,q}(\R^N)\hookrightarrow L^{q_\eta^{**}}_\eta(\R^N).
\)
\end{enumroman}
\noindent In particular, we have that
\[
\mathcal W\hookrightarrow L^q_{2q}(\R^N)\quad\text{and}\quad \mathcal W\hookrightarrow L^{p}_p(\R^N).
\]
\end{lemma}

\begin{proof}
Item (i) is the classical Hardy--Sobolev inequality.
For item (ii), observe first that the endpoint case
\(\eta=2q\) follows from the Hardy--Rellich inequality (see, for example,
\cite{MR1612685}), namely,
\[
\int_{\R^N}\frac{|u|^q}{|x|^{2q}}\,dx
\le
C\int_{\R^N}|\Delta u|^q\,dx,
\qquad
u\in D^{2,q}(\R^N).
\]
On the other hand, the endpoint case \(\eta=0\) is precisely the classical
second-order Sobolev embedding $D^{2,q}(\R^N)\hookrightarrow L^{q^{**}}(\R^N).$ The remaining values \(0<\eta<2q\) are obtained by interpolation between these two endpoint inequalities. Indeed, writing $\eta=2q\theta,$ $0\le\theta\le1,$
one has
\[
\frac1{q_\eta^{**}}
=
\frac{\theta}{q}
+
\frac{1-\theta}{q^{**}},
\]
and Hölder's inequality yields
\[
\left(
\int_{\R^N}
\frac{|u|^{q_\eta^{**}}}{|x|^\eta}\,dx
\right)^{1/q_\eta^{**}}
\le
\left(
\int_{\R^N}
\frac{|u|^q}{|x|^{2q}}\,dx
\right)^{\theta/q}
\left(
\int_{\R^N}
|u|^{q^{**}}\,dx
\right)^{(1-\theta)/q^{**}}.
\]
The conclusion now follows from the two endpoint estimates.
\end{proof}

The following lemma provides the weighted embeddings suggested by the Caffarelli--Kohn--Nirenberg inequalities. Here we use the notation
$q'={q}/({q-1})$.

\begin{lemma}\label{lem:weighted-embedding-E}
Assume \eqref{p,q,q*}
and let
\(0\le \eta<\overline{\eta}:=2+{N}/{q'}\).
Then the space $\mathcal V$
is continuously embedded into \(L_\eta^r(\R^N)\) for every
\[
1\le r\in(p_\eta^*,q_\eta^{**}).
\]
That is, for every such \(r\) there exists a constant \(C=C(N,p,q,\eta,r)>0\) such that
\begin{equation}\label{CKN cons}
\|u\|_{L_\eta^r}\le C\|u\|_{\mathcal V}
\qquad
\text{for all }u\in \mathcal V.
\end{equation}
Moreover, if \(\eta>0\), then the embedding
\(
\mathcal V\hookrightarrow L_\eta^r(\R^N)
\)
is compact for any
\(
1\le r\in(p_\eta^*,q_\eta^{**}).
\)
\end{lemma}
\begin{proof}
For $p,q$ as in \eqref{p,q,q*}, we fix constants $\eta\in[0,\overline{\eta})$, $1\le r\in(p_\eta^*,q_\eta^{**})$ and $\theta\in(0,1)$ such that
\begin{align*}
    \frac{1}{r}=\frac{\theta}{p^*_\eta}+\frac{1-\theta}{q^{**}_\eta}.
\end{align*}
Then CKN inequality (see \cite{MR0768824}) implies the existence of $C>0$ such that
\begin{equation*}
    \|u\|_{L^r_\eta}\le C\|\nabla u\|_p^\theta\|u\|_{q^{**}}^{1-\theta}\leq C\|u\|_{\mathcal V},\quad\forall u\in C^\infty_c(\mathbb{R}^N).
\end{equation*}
Using that $C^\infty_c(\mathbb{R}^N)$ is dense in $\mathcal V$, we conclude that \eqref{CKN cons} holds. Moreover, due to \cite[Proposition 24]{MR4659077} we obtain the compactness of the embedding in case $\eta\in(0,\overline\eta)$.
\end{proof}

\begin{remark}
    In the radial case,  we can allow $\eta=0$. Due to \cite[Proposition 26]{MR4659077} we obtain the compactness of the embedding of $\mathcal V_{rad} \hookrightarrow L^r(\R^N)$ for $r\in(p^*,q^{**})$.
\end{remark}

\begin{remark}
Since
\(
\mathcal W=D^{1,p}(\R^N)\cap D^{2,q}(\R^N)
\hookrightarrow \mathcal V
\)
continuously, Lemma~\ref{lem:weighted-embedding-E} immediately yields the same continuous and compact embeddings for \(\mathcal W\).
\end{remark}

\begin{corollary}\label{cor compact emb}
Assume \eqref{p,q,q*}.    
 For $\sigma$ and $r_\sigma$ as in \eqref{sigma_and_rsigma} we get $1\le r_\sigma\in(p_\sigma^*,q_\sigma^{**})$. 
In particular, the embedding \(\mathcal W\hookrightarrow L_\sigma^{r_\sigma}(\R^N)\) is compact. 
\end{corollary}


\begin{figure}[!htbp]
    \centering
    \includegraphics[width=0.5\textwidth]{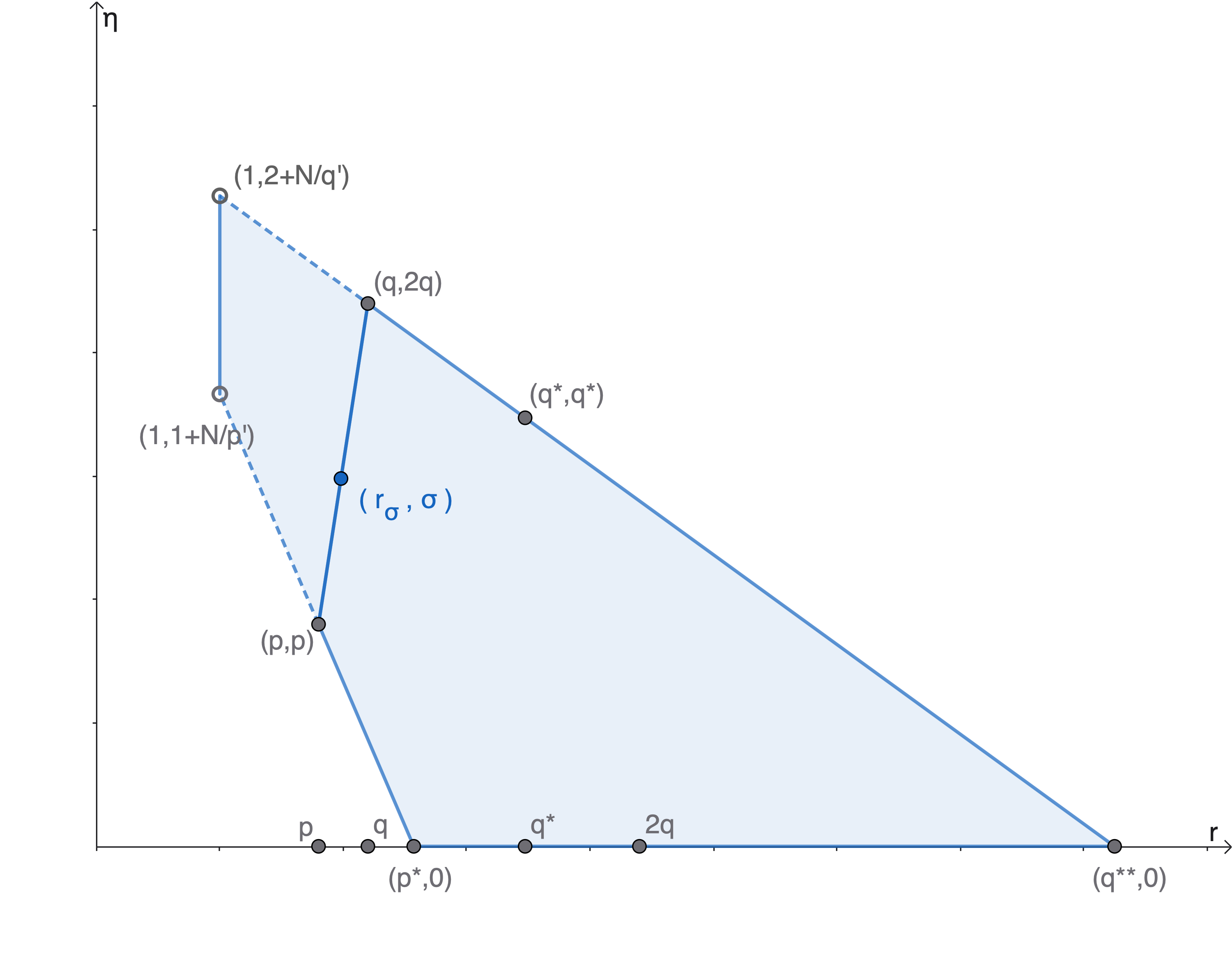}
    \caption{Embedding region in case $1<p<q<p^*$}
    \label{Embedings}
\end{figure}

\section{Scaling and scaled operators}\label{scaled theory}

This section will present a portion of a novel critical point theory for scaled operators on a Banach space. A detailed presentation of this technique can be found in \cite{MR5043800}. 

\begin{definition}\label{scaling defi}
    Let $\mathcal W$ be a reflexive Banach space. A scaling on $\mathcal W$ is a continuous mapping $\mathcal W\times[0,\infty)\to \mathcal W$, $(u,t)\mapsto u_t$ satisfying
    \begin{enumerate}[label=$(H_{\arabic*})$]
    \item\label{H1} $(u_{t_1})_{t_2}=u_{t_1t_2}$ for all $u\in \mathcal W$ and $t_1,t_2\geq0$;
    \item  $(\tau u)_t=\tau u_t$, for all $u\in \mathcal W$, $\tau\in\R$ and $t\geq0$;
    \item \label{H3} $u_0=0$ and $u_1=u$ for all $u\in \mathcal W$;
    \item \label{H4} $u_t$ is bounded on bounded sets of $\mathcal W\times[0,\infty)$;
    \item \label{H5} $\exists s>0$ such that $\|u_t\|=O(t^{s})$ as $t\to\infty$ uniformly on bounded sets.
\end{enumerate}
\end{definition}
 Denote by $\mathcal W^*$ the dual of $\mathcal W$. Recall that $h\in C(\mathcal W,\mathcal W^*)$ is a potential operator if there is a functional $H\in C^1(\mathcal W,\R)$, called a potential for $h$, such that $H^\prime=h$. By replacing $H$ with $H-H(0)$ if necessary, we may assume that $H(0)=0$.

 \begin{definition}\label{def scaled op}
     A scaled operator is an odd potential operator $h\in C(\mathcal W,\mathcal W^*)$ that maps bounded sets into bounded sets and satisfies
\[h(u_t)v_t=t^sh(u)v,\quad\forall u,v\in \mathcal W, t\geq0.
     \]
     \end{definition}
Let us denote by $I_s$ the potential of the scaled operator $A_s$ with $I_s(0)=0$. Of course $I_s$ is even, bounded on bounded sets, and satisfies the scaling property (see \cite[Proposition 2.2]{MR5043800})
\[I_s(u_t)=t^sI_s(u),\quad\forall u\in \mathcal W, t\geq0.
\]

\noindent We will consider the question of existence and multiplicity of solutions to equations as
\[
A_s(u)=\tilde{f}(u)\quad\text{in } \mathcal W^*,
\]
where $\tilde{f}\in C(\mathcal W,\mathcal W^*)$ is a potential operator, and $A_s$ is a scaled operator satisfying
\begin{enumerate}[label=$(H_{\arabic*})$, start=6]
\item \label{H6} $A_s(u)u>0$ for all $u\in \mathcal W\backslash\set{0}$;
\item \label{H7} every sequence $(u_j)$ in $\mathcal W$ such that $u_j\weak u$ and $A_s(u_j)(u_j-u)\to0$ has a subsequence that converges strongly to $u$.
\end{enumerate}
 Solutions of the above equation coincide with critical points of the $C^1$-functional
\[
\Phi(u)=I_s(u)-\tilde{F}(u),\quad u\in \mathcal W,
\]
where $\tilde{F}$ is the potential of $\tilde{f}$ with $\tilde{F}(0)=0$.

\subsection{Scaled eigenvalue problems}
Now let us consider the eigenvalue problem
 \begin{equation}\label{3-1}
     A_s(u)=\lambda B_s(u)\quad\text{in } \mathcal W^*,
 \end{equation}
where $\lambda\in\R$ and $A_s$ and $B_s$ are scaled operators satisfying $(H_6)$ and $(H_7)$ and
\begin{enumerate}[label=$(H_{\arabic*})$, start=8]
    \item \label{H8} $B_s(u)u>0$ for all $u\in \mathcal W\backslash\set{0}$;
    \item \label{H9} if $u_j\weak u$ in $\mathcal W$, then $B_s(u_j)\to B_s(u)$ in $\mathcal W^*$.
\end{enumerate}
 We say that $\lambda$ is an eigenvalue if there is a $u\in \mathcal W\backslash\set{0}$, called an eigenfunction associated with $\lambda$, satisfying equation $\eqref{3-1}$. If that is the case, then $u_t$ is also an eigenfunction associated with $\lambda$ for any $t>0$ since
\[
A_s(u_t)v=A_s(u_t)(v_{1/t})_t=t^sA_s(u)v_{1/t}=t^s\lambda B_s(u)v_{1/t}=\lambda B_s(u_t)(v_{1/t})_t=\lambda B_s(u_t)v
\]
for all $v\in \mathcal W$. We denote by $\sigma(A_s,B_s)$ the spectrum, i.e., the set of all eigenvalues, of the pair of scaled operators $(A_s,B_s)$. We have $\sigma(A_s,B_s)\subset (0,\infty)$ by $(H_6)$ and $(H_8)$.

Let us denote by $J_s$ the potential of $B_s$ with $J_s(0)=0$. We also assume that $I_s$ and $J_s$ satisfy
\begin{enumerate}[label=$(H_{\arabic*})$, start=10]
    \item\label{H10} $I_s$ is coercive, i.e., $I_s(u)\to\infty$ as $\|u\|\to\infty$;
    \item \label{H11} the equation $I_s(tu)=1$ has a unique positive solution $t$ for each $u\in \mathcal W\backslash\set{0}$;
    \item \label{H12} every solution of $\eqref{3-1}$ satisfies $I_s(u)=\lambda J_s(u).$
\end{enumerate}

The eigenvalue problem $\eqref{3-1}$ has the following variational formulation. Let
\[
\Psi(u)=\frac{1}{J_s(u)},\quad u\in \mathcal W\backslash\set{0},
\]
and consider
\[
\M=\set{u\in \mathcal W:I_s(u)=1}.
\]
 Then $\M$ is a complete, symmetric, and bounded $C^1$-Finsler manifold, and eigenvalues of problem $\eqref{3-1}$ coincide with critical values of $\w{\Psi}:=\Psi|_{\M}$. (see \cite[Proposition 2.5]{MR5043800}).

\subsubsection{Minimax eigenvalues}
We denote by $i(A)$ the $\mathbb{Z}_2$-cohomological index of a
symmetric subset $A\subset \mathcal W\setminus\{0\}$, introduced by Fadell and Rabinowitz. We
refer to~\cite{MR0478189} for its definition and basic properties.

Let $\F$ denote the class of symmetric subsets of $\M$. For $k\geq1$, let
\[
\F_k=\set{M\in\F:i(M)\geq k}
\]
and set
\[
\lambda_k:=\inf_{M\in\F_k}\sup_{u\in M}\w{\Psi}(u).
\]
We have the following theorem (see Perera et al. \cite[ Proposition 3.52 and Proposition 3.53]{MR2640827}).
\begin{theorem}\label{Theorem 301}
    Assume $(H_1)-(H_{12})$. Then $\lambda_k\nearrow \infty$ is a sequence of eigenvalues of $\eqref{3-1}$.
    \begin{enumroman}
        \item The first eigenvalue is given by
        \[
        \lambda_1=\min_{u\in\M}\w{\Psi}(u)>0.
        \]
        \item If $\lambda_k=\dotsb=\lambda_{k+m-1}=\lambda$, and $\mathcal E_\lambda$ is the set of eigenfunctions associated with $\lambda$ that lie on $\M$, then $i(\mathcal E_\lambda)\geq m$.
        \item\label{301-3} If $\lambda_k<\lambda<\lambda_{k+1}$, then
\[i(\w{\Psi}^{\lambda_k})=i(\M\backslash\w{\Psi}_\lambda)=i(\w{\Psi}^\lambda)=i(\M\backslash\w{\Psi}_{\lambda_{k+1}})=k,\]
        where $\w{\Psi}^a=\set{u\in\M:\w{\Psi}(u)\leq a}$ and $\w{\Psi}_a=\set{u\in\M:\w{\Psi}(u)\geq a}$ for $a\in\R$.
    \end{enumroman}
\end{theorem}


\subsection{Minimax values based on scaling}

Let $\Phi\in C^1(\mathcal W,\R)$ be an even functional, i.e., $\Phi(-u)=\Phi(u)$ for all $u\in \mathcal W$. 
Fix $c^*>0$ and let $\Gamma$ denote the group of odd homeomorphisms of $\mathcal W$ that are the identity outside $\Phi^{-1}(0,c^*)$. Let $\A^*$ denote the class of symmetric subsets of $\mathcal W$, and for $\rho>0$ let
\begin{eqnarray*}
    \M_\rho=\set{u\in \mathcal W: I_s(u)=\rho^s}=\set{u_\rho:u\in\M}.
\end{eqnarray*}

\begin{definition}[Benci \cite{MR84c:58014}]
    The pseudo-index of $M\in\A^*$ related to $i$, $\M_\rho$, and $\Gamma$ is defined by 
    \begin{eqnarray*}
        i^*(M)=\min_{\gamma\in\Gamma} i(\gamma(M)\cap\M_\rho).
    \end{eqnarray*}
\end{definition}

The main critical point theorem is stated below. 

\begin{theorem}[{\cite[Theorem 2.33]{MR5043800}}]\label{new m1}
Let $\Phi\in C^1(\mathcal W,\R)$ be an even functional and assume that there exists $c^*>0$ such that $\Phi$ satisfies the $(PS)_c$ condition for every $c\in(0,c^*)$. Let $A_0$ and $B_0$ be symmetric subsets of $\M$ such that $A_0$ is compact, $B_0$ is closed, and
\begin{equation*}
i(A_0)\geq k+m-1,\qquad i(\M\backslash B_0)\leq k-1
\end{equation*}
for some $k,m\geq1$. Let $R>\rho>0$ and set
\[
X=\set{u_t:u\in A_0,\ 0\leq t\leq R},\qquad A=\set{u_R:u\in A_0},\qquad B=\set{u_\rho:u\in B_0}.
\]
Assume that
\begin{equation}\label{new m1-2}
\sup_{u\in A}\Phi(u)\leq0<\inf_{u\in B}\Phi(u),\qquad \sup_{u\in X}\Phi(u)<c^*.
\end{equation}
For $j=k,\ldots,k+m-1$, let
\[
\mathcal A_j^*=\set{M\in\mathcal A^*:M\text{ is compact and }i^*(M)\geq j}
\]
and set
\[
c_j^*=\inf_{M\in\mathcal A_j^*}\max_{u\in M}\Phi(u).
\]
Then
\(
0<c_k^*\leq\cdots\leq c_{k+m-1}^*<c^*,
\)
each $c_j^*$ is a critical value of $\Phi$, and $\Phi$ has $m$ distinct pairs of associated critical points.
\end{theorem}

For $k=1$ this can be seen as a Mountain-Pass type result.

\subsection{Scaling for the q-Biharmonic and p-Laplacian type problem}

We consider the space $\mathcal W=D^{1,p}(\R^N)\cap\,D^{2,q}(\R^N) $, which is a reflexive Banach space with the norm
\[
\|u\|=\|\Delta u\|_q+\|\nabla u\|_p.
\]
Now we are going to present a scaling on $\mathcal W$, according to Definition \ref{scaling defi}. For $p\neq q,q^*$ we denote
\[
\gamma=\frac{2q-p}{p-q}
\]
and define the map $(u,t)\in \mathcal W\times[0,\infty)\mapsto u_t\in \mathcal W$ by $u_0\equiv0$ and for $t>0$
\begin{equation}\label{ut definition}
    \begin{aligned}
    &u_t(x)=t^\gamma u(tx)\quad\text{if}\,\, p\in (q,q^*)\\
     & u_t(x)=t^{-\gamma} u(x/t)\quad\text{if}\,\, p\in (1,q).
\end{aligned}
\end{equation}

\begin{lemma} 
Assume $p$ and $q$ as in \eqref{p,q,q*}. Then, the operator $ \mathcal W\times[0,\infty)\to \mathcal W$, $(u,t)\mapsto u_t$, defined in \eqref{ut definition} is a \textit{scaling} on $\mathcal W$. 
\end{lemma}
\begin{proof}
It is clear that \ref{H1}-\ref{H3} are satisfied.   Now, for any $t>0$ we see that
\[
\int_{\R^N}|\Delta [t^\gamma u(tx)]|^qdx=t^{q(\gamma+2)-N}\int_{\R^N}|\Delta u|^qdx,
\]
and
\[
\int_{\R^N}|\nabla [t^\gamma u(tx)]|^pdx=t^{p(\gamma+1)-N}\int_{\R^N}|\nabla u|^pdx.
\]
By the choice of $\gamma$, it occurs
\[
q(\gamma+2)-N=p(\gamma+1)-N=\dfrac{pq}{p-q}-N=:s_0.
\]
Notice that $s_0>0$ for $p\in (q,q^*)$ and $s_0<0$ for $p\in (1,q)$.
Denoting $s=|s_0|$, we get
\[
\|\Delta u_t\|_q^q=t^s\|\Delta u\|_q^q\quad\text{and}\quad \|\nabla u_t\|_p^p= t^s\|\nabla u\|_p^p,
\]
and so
\begin{equation}\label{ut bounded}
\|u_t\|\leq\max\{t^{s/q},t^{s/p}\}\| u\|,\quad\forall t>0, u\in \mathcal W.
\end{equation}
This shows that $(u,t)\mapsto u_t$ is continuous in $(u,0)$. If  $t>0$, using the density of $C^\infty_c(\mathbb{R}^N)$ in $\mathcal W$ and arguing as in \cite[Lemma 3.1]{MR5043800} we verify the continuity in $(u,t)$. From \eqref{ut bounded} we also see that $u_t$ is bounded on bounded sets of $\mathcal W\times[0,\infty)$ and \ref{H4} holds. Moreover, since $p,q>1$, we have $\|u_t\|=o(t^s)$ as $t\to\infty$, uniformly on bounded sets of $\mathcal W$. In particular, 
\ref{H5} also holds.
\end{proof} 

Now that we have a scaling on $\mathcal W$, let us define the potential operators associated to the problems we are studying here. 
Consider $A_s:\mathcal W\to \mathcal W^*$  defined by
\begin{equation}\label{As def}
A_s(u)v=\int_{\R^N}|\Delta u|^{q-2}\Delta u \Delta vdx+\int_{\R^N}|\nabla u|^{p-2}\nabla u \nabla vdx
\end{equation}
and, for a fixed $p<\sigma<2q$ and $r_\sigma$ as in \eqref{sigma_and_rsigma}, $B_s:\mathcal W\to \mathcal W^*$ is given by
\begin{equation}\label{Bs def}
B_s(u)v=\int_{\R^N}\dfrac{|u|^{r_\sigma-2}uv}{|x|^\sigma}dx.
\end{equation}
$A_s$ and $B_s$ are odd potential operators, map bounded sets into bounded sets and, by the definition of $u_t$ in \eqref{ut definition}, they verify 
\[
A_s(u_t)v_t=t^sA_s(u)v\quad\text{and}\quad B_s(u_t)v_t=t^sB_s(u)v\quad\forall  u,v\in \mathcal W,t\ge0.
\]
This means that $A_s$ and $B_s$ are scaled operators, as in Definition \ref{def scaled op}. The associated potentials $I_s,J_s:\mathcal W\to\R$, such that $I_s(0)=0$ and $J_s(0)=0$, are given by
\begin{equation}\label {Is Js def}
    I_s(u)=\int_{\R^N}\seq{\frac{1}{q}|\Delta u|^q+\frac{1}{p}|\nabla u|^p}dx
    \quad\text{and}\quad J_s(u)=\frac{1}{r_\sigma}\int_{\R^N}\dfrac{|u|^{r_\sigma}}{|x|^\sigma}dx
\end{equation}
and satisfy
\[
I_s(u_t)=t^s I_s(u)\quad\text{and}\quad J_s(u_t)=t^s J_s(u),\quad\forall t\ge0.
\]
\begin{lemma}\label{S+ condition}
Every sequence $(u_j)$ in $\mathcal W$ such that $u_j\weak u$ and $\limsup_{j\to\infty}A_s(u_j)(u_j-u)\le0$ has a subsequence that converges strongly to $u$.
\end{lemma}
\begin{proof}
Since $A_s(u)\in \mathcal W^*$ and $u_j\weak u$ in $\mathcal W$, we have $A_s(u)u_j\to A_s(u)u$. Moreover, using that
\begin{align}\label{basic ineq}
 (|z|^{r-2}z-|\xi|^{r-2}\xi)(z-\xi)\ge
\begin{cases}
 C_r|z-\xi|^r&\text{if}\quad r\ge2,\\
  \frac{|z-\xi|^2}{(|z|+|\xi|)^{2-r}}&\text{if}\quad r\in(1,2],
\end{cases}
\end{align}
for all $z,\xi\in\R^m$, with $m=1$ and $m=N$, we get
\[
0\le \limsup_{j\to\infty}\langle A_s(u_j)- A_s(u),u_j-u\rangle\leq0.
\]
This implies that $\langle A_s(u_j)- A_s(u),u_j-u\rangle\to0$,
\[
\int_{\R^N}\left[|\Delta u_j|^{q-2}\Delta u_j -|\Delta u|^{q-2}\Delta u\right]\Delta (u_j-u)dx\to 0
\]
and 
\[
\int_{\R^N}\left[|\nabla u_j|^{p-2}\nabla u_j-|\nabla u|^{p-2}\nabla u\right] \nabla (u_j-u)dx\to0.
\]
Then, from \eqref{basic ineq} we  get $u_j\to u$ in $\mathcal W$. 
\end{proof}
\begin{lemma}
  Assume \eqref{p,q,q*}. Then, the scaled operators $A_s$ and $B_s$, defined in \eqref{As def} and  \eqref{Bs def}, satisfy \ref{H6}-\ref{H9} and their potentials $I_s$ and $J_s$, given in \eqref{Is Js def}, satisfy \ref{H10}-\ref{H12}.
\end{lemma}
\begin{proof}
Conditions \ref{H6} and \ref{H8} follow from the definition of $A_s$ and $B_s$. Lemma \ref{S+ condition} gives \ref{H7} and Corollary \ref{cor compact emb} yields \ref{H9}. We also verify that there is a positive constant $c=c({p,q})$ such that 
$$
I_s(u)\geq c\|u\|^{\min\{p,q\}},\quad \forall u\in \mathcal W\,\,\text{with}\,\,\,I_s(u)\ge1,
$$
and so condition \ref{H10} is satisfied. Now, for any $u\neq0$, the function $t\mapsto I_s(tu)$ is increasing on $[0,\infty)$, and so \ref{H11} holds.
Let us  prove \ref{H12}. If $u\in \mathcal W$ is an eigenfunction associated with $\lambda$, using $u$ as a test function in 
 \eqref{intro:eigenvalue-problem} gives
\begin{equation}\label{nehari-H12}
\int_{\R^N}|\Delta u|^q\,dx+\int_{\R^N}|\nabla u|^p\,dx=\lambda \int_{\R^N}\frac{|u|^{r_\sigma}}{|x|^\sigma}\,dx.
\end{equation}
Now, we notice that since $u$ is a weak solution of \eqref{intro:eigenvalue-problem}, taking $\Omega=\mathbb R^N\setminus\{0\}$ and $f(x,u)=\lambda {|u|^{r_\sigma-2}u}/{|x|^\sigma}$, we have that $u$ has the regularity given in \eqref{all estimate}. Therefore, $u$ satisfies the Pohozaev identity (Proposition \ref{prop:pohozaev-identity} with $\tau_1,\tau_2=0$), which yields
\begin{equation}\label{pohozaev-H12-proof}
\frac{N-2q}{q}\int_{\R^N}|\Delta u|^q\,dx
+
\frac{N-p}{p}\int_{\R^N}|\nabla u|^p\,dx
=
\lambda\frac{N-\sigma}{r_\sigma}\int_{\R^N}\frac{|u|^{r_\sigma}}{|x|^\sigma}\,dx.
\end{equation}
Using \eqref{nehari-H12} to eliminate $\lambda \int_{\R^N}\frac{|u|^{r_\sigma}}{|x|^\sigma}\,dx$ from
\eqref{pohozaev-H12-proof}, we obtain
\begin{equation}\label{linear-relation-H12}
\left(
\frac{N-2q}{q}-\frac{N-\sigma}{r_\sigma}
\right)\int_{\R^N}|\Delta u|^q\,dx
+
\left(
\frac{N-p}{p}-\frac{N-\sigma}{r_\sigma}
\right)\int_{\R^N}|\nabla u|^p\,dx
=0.
\end{equation}
Since \(r_\sigma(2q-p)=pq+\sigma(q-p)\), we verify that  the matrix with the two vectors
\[
\left(
\frac1q-\frac1{r_\sigma},\,
\frac1p-\frac1{r_\sigma}
\right)\quad\text{and}\quad
\left(
\frac{N-2q}{q}-\frac{N-\sigma}{r_\sigma},\,
\frac{N-p}{p}-\frac{N-\sigma}{r_\sigma}
\right),
\]
has trivial determinant. Moreover, recalling that $p\neq q^*$, we see that the second vector is nonzero.  
It follows from \eqref{linear-relation-H12} and the proportionality of the
two vectors that
\[
\left(\frac1q-\frac1{r_\sigma}\right)\int_{\R^N}|\Delta u|^q\,dx
+
\left(\frac1p-\frac1{r_\sigma}\right)\int_{\R^N}|\nabla u|^p\,dx
=0.
\]
Using \eqref{nehari-H12}, we conclude that
\begin{align*}
\frac1q \int_{\R^N}|\Delta u|^q\,dx+\frac1p \int_{\R^N}|\nabla u|^p\,dx-\frac{\lambda}{r_\sigma}
\int_{\R^N}\frac{|u|^{r_\sigma}}{|x|^\sigma}\,dx=0.
\end{align*}
Therefore,
$I_s(u)=\lambda J_s(u)$,
which is precisely condition \ref{H12}.
\end{proof}

\section{Concentration--compactness tools}\label{section-conccomp}

In this section, we establish the concentration--compactness results that will
be used to prove some pointwise convergence of $(PS)$ sequences. The arguments are formulated so as to cover simultaneously the two regimes $q<p^*$ and $q>p^*.$ In particular, the results below apply both to the full double-critical problem and to the reduced problem obtained by removing the $p^*$-critical term.

We denote by
\begin{equation}\label{S1pS2q}
S_{1,p}
:=
\inf_{u\in D^{1,p}(\mathbb R^N)\backslash\{0\}}
\frac{\displaystyle\int_{\mathbb R^N}|\nabla u|^p\,dx}
{\displaystyle
\left(
\int_{\mathbb R^N}|u|^{p^*}\,dx
\right)^{p/p^*}},\qquad S_{2,q}
:=
\inf_{u\in D^{2,q}(\mathbb R^N)\backslash\{0\}}
\frac{\displaystyle\int_{\mathbb R^N}|\Delta u|^q\,dx}
{\displaystyle
\left(
\int_{\mathbb R^N}|u|^{q^{**}}\,dx
\right)^{q/q^{**}}}
\end{equation}
the optimal constants associated with the  Sobolev
embeddings $ D^{1,p}(\mathbb R^N)
\hookrightarrow
L^{p^*}(\mathbb R^N)$ and  $ D^{2,q}(\mathbb R^N)
\hookrightarrow
L^{q^{**}}(\mathbb R^N),$ respectively.

\medskip

Next, we have the concentration--compactness lemma obtained by Lions \cite[Lemma I.1]{MR834360}. 

\begin{lemma}\label{lem:Lions} Let $\{u_j\}\subset D^{2,q}(\mathbb R^N)$ be a bounded
sequence such that $u_j\rightharpoonup u$
 in $ D^{2,q}(\mathbb R^N).$
Suppose that, up to a subsequence, \[|\Delta u_j|^q\,dx
\stackrel{*}{\rightharpoonup}\mu_q
 \qquad\text{and}\qquad |u_j|^{q^{**}}\,dx \stackrel{*}{\rightharpoonup}\nu_q\] in the sense of bounded nonnegative measures on $\mathbb R^N$.
Then there exist an at most countable index set $\J_q$, a family of
distinct points $\{x_{q,k}\}_{k\in \J_q}\subset\mathbb R^N$ and numbers
$\{\nu_{q,k}\}_{k\in \J_q}\subset(0,\infty)$ such that
\[
\nu_q
=
|u|^{q^{**}}\,dx
+
\sum_{k\in \J_q}\nu_{q,k}\delta_{x_{q,k}}.
\]
Moreover, there exist  $\mu_{q,k}>0$, such that
\[
\mu_q
\ge
|\Delta u|^q\,dx
+
\sum_{k\in \J_q}\mu_{q,k}\delta_{x_{q,k}},
\]
and
\begin{equation*}
(\nu_{q,k})^{q/q^{**}}
\le
S_{2,q}^{-1}\mu_{q,k}
\qquad\text{for every }k\in \J_q.
\end{equation*}
In particular,
\[
\sum_{k\in \J_q}(\nu_{q,k})^{q/q^{**}}<\infty.
\]
\end{lemma}

\begin{proof}
 The  assertion follows from 
the same localization argument of the classical concentration--compactness lemma of Lions \cite[Lemma I.1]{MR834360}, applied to the critical Sobolev
embedding of $ D^{2,q}(\mathbb R^N)$ endowed with the equivalent
norm $\|\Delta u\|_{q}$, and hence with the corresponding optimal constant $S_{2,q}$.
\end{proof}

For the rest of this section, let us unify the functional associated to either problem \eqref{prob1} or problem \eqref{intro:problem-p-less-q-large} in the following manner:

Let
\[
\tau:=
\begin{cases}
0, & \text{if } q>p^*,\\
1, & \text{if } q<p^*.
\end{cases}
\]
We define the functional \(\Phi_\tau\in C^1(\mathcal W,\mathbb{R})\) by
\begin{align}\label{energy-tau}
\Phi_\tau(u)
&:=
\frac{1}{q}\int_{\mathbb{R}^N}|\Delta u|^q\,dx
+\frac{1}{p}\int_{\mathbb{R}^N}|\nabla u|^p\,dx
-\frac{\lambda}{r_\sigma}
\int_{\mathbb{R}^N}\frac{|u|^{r_\sigma}}{|x|^\sigma}\,dx
\nonumber\\
&\quad
-\frac{1}{q^{**}}\int_{\mathbb{R}^N}|u|^{q^{**}}\,dx
-\frac{\tau}{p^*}\int_{\mathbb{R}^N}|u|^{p^*}\,dx,
\qquad u\in \mathcal W.
\end{align}

\begin{lemma}\label{r3}
    Let $\set{u_j}_{j\in\N}\subset \mathcal W$ be a bounded $(PS)$ sequence for $\Phi_\tau$. Then  the family $\J_q$ given in Lemma \ref{lem:Lions} is  finite.
\end{lemma}
\begin{proof}
Since $\mathcal W$ is a reflexive Banach space and  $\set{u_j}_{j\in\N}\subset \mathcal W$ is bounded, there is $u\in \mathcal W$ such that $u_j\rightharpoonup u$ weakly in $ \mathcal W$, up to a subsequence. Then Lemma \ref{lem:Lions} can be applied. Let us fix $k\in \J_q$ and $x_{q,k}$ the corresponding point in $\mathbb R^N$, given by Lemma \ref{lem:Lions}. Let $\varphi\in C_c^\infty(\R^N)$ such that $0\leq\varphi\leq1$ and
    \begin{eqnarray*}
          \varphi(x)=
        1\quad\text{if }\,\,\,x\in B_1(0),\qquad
       \varphi(x)= 0\quad\text{if }\,\,\,x\in \R^N\backslash B_2(0).
    \end{eqnarray*}
    For $\theta>0$, let us also define 
    \[
    \varphi_\theta(x)=\varphi\seq{\frac{x-x_{q,k}}{\theta}}.
    \]
    It is clear that there exists a constant $C>0$, independent of $\theta$, such that $|\nabla\varphi_\theta(x)|\leq C/\theta$ and $|\Delta \varphi_\theta(x)|\leq C/\theta^2$ for all $x\in\R^N$. Notice that $\varphi_\theta(x)\to\mathcal{X}_{\set{x_{q,k}}}(x)$ as $\theta\to0$ for each $x\in\R^N$, where $\mathcal{X}_A$ is the characteristic function of the set $A\subset\R^N$.
Moreover, for each  fixed $\theta>0$, one has that $\set{\varphi_\theta u_j}_{j\in\N}$ is a bounded sequence in $\mathcal W$. So  
    we have $\Phi^\prime_\tau(u_j)(\varphi_\theta u_j)=o(1)$ as $j\to\infty$, which means that
    \begin{align}\label{r3-1}
        o(1)=&\int_{\R^N}|\Delta u_j|^q\varphi_\theta\,dx-\int_{\R^N}|u_j|^{q^{**}}\varphi_\theta\,dx+
        \int_{\R^N}|\nabla u_j|^p\varphi_\theta\,dx-\tau\int_{\R^N}|u_j|^{p^*}\varphi_\theta\,dx\nonumber\\
        &-\lambda\int_{\R^N}\frac{|u_j|^{r_\sigma}}{|x|^\sigma}\varphi_\theta\,dx
        +\int_{\R^N}u_j|\Delta u_j|^{q-2}\Delta u_j\Delta \varphi_\theta \,dx\nonumber\\
        &+2\int_{\R^N}|\Delta u_j|^{q-2}\Delta u_j\nabla\varphi_\theta\cdot\nabla u_j\,dx +\int_{\R^N}u_j|\nabla u_j|^{p-2}\nabla u_j\cdot\nabla\varphi_\theta \,dx.
    \end{align}
    Since $u_j\rightharpoonup u$ in $\mathcal W$, and the embedding $\mathcal W\hookrightarrow L^{r_\sigma}_\sigma(\R^N)$ is compact, then $u_j\to u$ in $L^{r_\sigma}_\sigma(\R^N)$ and
    \[
\lim_{j\to\infty}\int_{\R^N}\frac{|u_j|^{r_\sigma}}{|x|^\sigma}\varphi_\theta\,dx
=\int_{\R^N}\frac{|u|^{r_\sigma}}{|x|^\sigma}\varphi_\theta\,dx.
    \]
Notice that, since we have compactness of the embeddings  $\mathcal W\hookrightarrow L^{p^*}_{loc}(\R^N)$ and $W^{1,p}_{loc}(\R^N)$ (due to $p<q^*$), we also get
\[\lim_{j\to\infty}\int_{\R^N}|\nabla u_j|^p\varphi_\theta\,dx
=\int_{\R^N}|\nabla u|^p\varphi_\theta\,dx\quad\text{and}\quad \lim_{j\to\infty}\int_{\R^N}|u_j|^{p^*}\varphi_\theta\,dx=\int_{\R^N}|u|^{p^*}\varphi_\theta\,dx.
\]
  Now, since $\varphi_\theta(x)\to 0$ a.e. on $\R^N$ as $\theta\searrow0$,  by Lebesgue's Dominated Convergence Theorem we obtain
    \begin{eqnarray}\label{r3-2}
\lim_{\theta\searrow0}\int_{\R^N}\frac{|u|^{r_\sigma}}{|x|^\sigma}\varphi_\theta\,dx=0,\ \ \lim_{\theta\searrow0}\int_{\R^N}|\nabla u|^p\varphi_\theta\,dx=0\ \ \text{and}\ \ \lim_{\theta\searrow0}\int_{\R^N}|u|^{p^*}\varphi_\theta\,dx=0. 
    \end{eqnarray}
We also have
\begin{eqnarray*}
    \left|\int_{\R^N}u_j|\Delta u_j|^{q-2}\Delta u_j\Delta\varphi_\theta \, dx\right|&\leq&\int_{\R^N}|\Delta u_j|^{q-1}|\Delta \varphi_\theta|^{1/2}|u_j||\Delta\varphi_\theta|^{1/2}\,dx\nonumber\\
    &\leq&\frac{1}{\theta^2}\seq{\int_{\R^N}|\Delta u_j|^{(q-1)(q^{**})^\prime}\left|\Delta\varphi\seq{\frac{x-x_{q,k}}{\theta}}\right|^{(q^{**})^\prime/2}dx}^{1/(q^{**})^\prime}\nonumber\\
    &&\times\seq{\int_{\R^N}|u_j|^{q^{**}}\left|\Delta\varphi\seq{\frac{x-x_{q,k}}{\theta}}\right|^{q^{**}/2}dx}^{1/q^{**}}.\nonumber
    \end{eqnarray*}
  So, since $\{\Delta u_j\}$ is bounded in $L^q(\R^N)$ and $1/(q^{**})'-1/q'=2/N$, we get  
\begin{eqnarray*}    
   \left|\int_{\R^N}u_j|\Delta u_j|^{q-2}\Delta u_j\Delta\varphi_\theta \, dx\right|  &\leq&C\seq{\int_{\R^N}|u_j|^{q^{**}}\left|\Delta\varphi\seq{\frac{x-x_{q,k}}{\theta}}\right|^{q^{**}/2}}^{1/q^{**}}
\end{eqnarray*}
where $C>0$ is a constant independent of $j$ and $\theta$. Now, notice that
\begin{eqnarray*}
\int_{\R^N}|u_j|^{q^{**}}\left|\Delta\varphi\seq{\frac{x-x_{q,k}}{\theta}}\right|^{q^{**}/2}dx\to\int_{\R^N}\left|\Delta\varphi\seq{\frac{x-x_{q,k}}{\theta}}\right|^{q^{**}/2}d\nu_q
\end{eqnarray*}
as $j\to\infty$, and
\begin{eqnarray*}
\int_{\R^N}\left|\Delta\varphi\seq{\frac{x-x_{q,k}}{\theta}}\right|^{q^{**}/2}d\nu_q\to0
\end{eqnarray*}
as $\theta\searrow0$ by Lebesgue's theorem. Similarly, one can verify that
    \begin{eqnarray}\label{r3-6}
        \left|\int_{\R^N}u_j|\nabla u_j|^{p-2}\nabla u_j\cdot\nabla \varphi_\theta \,dx\right|\leq C \seq{\int_{\R^N}|u_j|^{p^*}\left|\nabla \varphi\seq{\frac{x-x_{q,k}}{\theta}}\right|^{p^*/2}dx}^{1/p^*}
    \end{eqnarray}
    for some constant $C>0$ independent of $j$ and $\theta$, and that
    \begin{eqnarray*}
        \int_{\R^N}|u_j|^{p^*}\left|\nabla \varphi\seq{\frac{x-x_{q,k}}{\theta}}\right|^{p^*/2}dx\to\int_{\R^N}|u|^{p^*}\left|\nabla \varphi\seq{\frac{x-x_{q,k}}{\theta}}\right|^{p^*/2}dx
    \end{eqnarray*}
    as $j\to\infty$, and
    \begin{eqnarray*}
\int_{\R^N}|u|^{p^*}\left|\nabla \varphi\seq{\frac{x-x_{q,k}}{\theta}}\right|^{p^*/2}dx\to0
    \end{eqnarray*}
    as $\theta\searrow0$.
    We also have
\begin{align*}
\left|
\int_{\mathbb R^N}
|\Delta u_j|^{q-2}\Delta u_j
\nabla\varphi_\theta\cdot\nabla u_j\,dx
\right|
&\le
C\left(
\frac1{\theta^q}
\int_{\mathbb R^N}
\left|\nabla\varphi\left(\frac{x-x_{q,k}}{\theta}\right)\right|^q
|\nabla u_j|^q\,dx
\right)^{1/q}.
\end{align*}
For every fixed $\theta>0$, the strong convergence
$\nabla u_j\to\nabla u$ in $L^q_{loc}(\R^N)$ gives
\[
\frac1{\theta^q}
\int_{\mathbb R^N}
\left|\nabla\varphi\left(\frac{x-x_{q,k}}{\theta}\right)\right|^q
|\nabla u_j|^q\,dx
\rightarrow
\frac1{\theta^q}
\int_{\mathbb R^N}
\left|\nabla\varphi\left(\frac{x-x_{q,k}}{\theta}\right)\right|^q
|\nabla u|^q\,dx.
\]
Using $1/q=1/q^*+1/N$ and arguing as in \eqref{r3-6}, we obtain
\begin{equation}\label{r3-8}
\left(
\frac1{\theta^q}
\int_{\mathbb R^N}
\left|\nabla\varphi\left(\frac{x-x_{q,k}}{\theta}\right)\right|^q
|\nabla u|^q\,dx
\right)^\frac{1}{q}
\le
C\left(
\int_{\mathbb R^N}
|\nabla u|^{q^*}
\left|\nabla\varphi\left(\frac{x-x_{q,k}}{\theta}\right)\right|^\frac{q^*}{2}
dx
\right)^\frac{1}{q^*}.
\end{equation}
The right-hand side tends to zero as $\theta\searrow0$ by the dominated
convergence theorem.

    Now, using all the estimates between \eqref{r3-2} and \eqref{r3-8}, and after taking the limit of \eqref{r3-1} first with $j
\to\infty$ and then $\theta\searrow0$, one gets
\begin{eqnarray}\label{r3-12}
    \nu_{q,k}=\mu_{q,k},\quad\forall k\in \J_q.
\end{eqnarray}
Suppose that $\#\J_q=\infty$. Then $\nu_{q,k}\to0$ as $k\to\infty$ by Lemma \ref{lem:Lions}. Let $k_0\in\N$ such that 
$\nu_{q,k}<1$ for all $k\geq k_0$. Then
\begin{eqnarray}\label{r3-13}
    \sum_{k\geq k_0}(\nu_{q,k})^{q/q^{**}}\geq\sum_{k\geq k_0}\nu_{q,k}= \sum_{k\geq k_0}\mu_{q,k}.
\end{eqnarray}

\noindent By Lemma~\ref{lem:Lions} and \eqref{r3-12},
\[
\mu_{q,k}=\nu_{q,k}
\le
S_{2,q}^{-q^{**}/q}(\mu_{q,k})^{q^{**}/q},
\]
thus, $\mu_{q,k}$ are uniformly bounded away from zero. If $\J_q$ were infinite, \eqref{r3-13}  would therefore imply
\[
\sum_{k\ge k_0}
(\nu_{q,k})^{q/q^{**}}
=\infty.
\]
which contradicts Lemma \ref{lem:Lions}. This shows that we have $\# \J_q<\infty$.
\end{proof}

\begin{lemma}\label{s0}
    Let $\set{u_j}_{j\in\N}\subset \mathcal{W}$ be a bounded $(PS)$ sequence of $\Phi_\tau$. Let $\set{x_{q,k}}_{k\in \J_q}\subset\R^N$ as in Lemma \ref{r3}. If $K\subset\subset\R^N\backslash\set{x_{q,k}:k\in \J_q}$, then $u_j\to u$ in $L^{q^{**}}(K)$.
\end{lemma}
\begin{proof}
    Identical to the proof contained in \cite[Lemma 3.6]{dos2010positive}.
\end{proof}

\begin{lemma}\label{mbs1}
    Let $\set{u_j}_{j\in\N}$ be a bounded $(PS)$ sequence of $\Phi_\tau$, and let $\set{x_{q,k}}_{k\in \J_q}$ be as Lemma \ref{r3}. If $K\subset\subset\R^N\backslash\set{x_{q,k}:k\in \J_q}$, then $\Delta u_j\to\Delta u$ in $L^{q}(K)$.
\end{lemma}
\begin{proof}
    Let $\delta=\text{dist}(K,\set{x_{q,k}:k\in \J_q})$. Notice that $\delta>0$ by Lemma \ref{r3}. For each $\theta\in(0,\delta)$, consider $N_\theta=\set{x\in\R^N:\text{dist}(x,K)<\theta}$ and $\xi_\theta\in C_c^\infty(\R^N)$, $0\leq\xi_\theta\leq1$ with
    \[
    \xi_\theta(x)=\begin{cases}
        1,&\quad\text{if }x\in N_{\theta/2},\\
        0,&\quad\text{if }x\in\R^N\backslash N_\theta.
    \end{cases}
    \]
    By inequality \eqref{basic ineq} we have
    \begin{eqnarray}\label{s1-1}
        0&\leq&\int_{K}\seq{|\Delta u_j|^{q-2}\Delta u_j-|\Delta u|^{q-2}\Delta u}(\Delta u_j-\Delta u)dx\nonumber\\
        &\leq&\int_{\R^N}\seq{|\Delta u_j|^{q-2}\Delta u_j-|\Delta u|^{q-2}\Delta u}(\Delta u_j-\Delta u)\xi_\theta dx\nonumber\\
        &=&\int_{\R^N}\seq{|\Delta u_j|^q\xi_\theta-|\Delta u_j|^{q-2}\Delta u_j\Delta u\xi_\theta-|\Delta u|^{q-2}\Delta u(\Delta u_j-\Delta u)\xi_\theta}dx.
    \end{eqnarray}
    Since $\Phi^\prime_\tau(u_j)\to0$ in $\mathcal W^*$ and $\set{\xi_\theta u_j}_{j\in\N}$ is bounded in $D^{2,q}(\R^N)\cap D^{1,p}(\R^N)$, we have
\[
\Phi_\tau'(u_j)\bigl(\xi_\theta u_j\bigr)
-
\Phi_\tau'(u_j)\bigl(\xi_\theta u\bigr)
=o(1).
\]
By the local compact embeddings, we have
\[
\nabla u_j\to\nabla u
\quad\text{in }(L^p_{\mathrm{loc}}(\R^N))^N,
\qquad
u_j\to u
\quad\text{in }L^{p^*}_{\mathrm{loc}}(\R^N).
\]
Moreover, Lemma \ref{s0} gives the local strong convergence in
\(L^{q^{**}}\) away from the concentration points, while the compactness of
the weighted embedding gives the convergence of the weighted term.
Consequently, all terms other than the \(q\)-biharmonic principal term are
\(o(1)\). Hence
    \begin{eqnarray*}
        \int_{\R^N}\seq{|\Delta u_j|^q\xi_\theta-|\Delta u_j|^{q-2}\Delta u_j\Delta u\xi_\theta}dx=\int_{\R^N}|\Delta u_j|^{q-2}\Delta u_j
\Delta\xi_\theta(u-u_j)dx\nonumber\\
+2\int_{\R^N}|\Delta u_j|^{q-2}\Delta u_j\nabla\xi_\theta\cdot\seq{\nabla u-\nabla u_j}dx+o(1).
\end{eqnarray*}
Now, define $A_\theta:=N_\theta\backslash\overline{N_{\theta/2}}$. By H\"older inequality one gets
\begin{eqnarray*}
    \int_{\R^N}|\Delta u_j|^{q-2}\Delta u_j
\Delta\xi_\theta(u-u_j)dx\leq C\seq{\int_{A_\theta}|u_j-u|^{q^{**}}dx}^{1/q^{**}}=o(1)
\end{eqnarray*}
by Lemma \ref{s0}. H\"older inequality again and $D^{2,q}(\mathbb R^N)\hookrightarrow W^{1,q}_{\mathrm{loc}}(\mathbb R^N)$ give us
\begin{eqnarray*}
    \left|\int_{\R^N}|\Delta u_j|^{q-2}\Delta u_j\nabla\xi_\theta\cdot\seq{\nabla u-\nabla u_j}dx\right|\leq C\seq{\int_{A_\theta}|\nabla u-\nabla u_j|^qdx}^{1/q}=o(1).
\end{eqnarray*}
 Also, since $\Delta u_j\rightharpoonup\Delta u$ in $L^q(\R^N)$, one also gets
\begin{eqnarray}\label{s1-5}
    \int_{\R^N}|\Delta u|^{q-2}\Delta u\xi_\theta(\Delta u_j-\Delta u)dx=o(1).
\end{eqnarray}
Now, combining the estimates from \eqref{s1-1} to \eqref{s1-5}  gives us
\[
\int_{K}\seq{|\Delta u_j|^{q-2}\Delta u_j-|\Delta u|^{q-2}\Delta u}(\Delta u_j-\Delta u)dx=o(1),
\]
and this together with inequality \eqref{basic ineq} yields
\[
\int_{K}|\Delta u_j-\Delta u|^qdx=o(1)
\]
and the conclusion follows.
\end{proof}

\begin{corollary}\label{s2}
    If $\set{u_j}_{j\in\N}$ is a bounded $(PS)$ sequence of $\Phi_\tau$, then there exists a subsequence, also denoted by $\set{u_j}_{j\in\N}$, such that the conclusions of Lemma \ref{lem:Lions} hold, with the addition that $\J_q$ is finite, and
    \begin{enumroman}
        \item $u_j\to u$ in $L^{q^{**}}(K)$ for every $K\subset\subset \R^N\backslash\set{x_{q,k}:k\in \J_q}$ and a.e. in $\mathbb R^N$,
        \item $\Delta u_j\to\Delta u$ in $L^{q}(K)$ for every $K\subset\subset \R^N\backslash\set{x_{q,k}:k\in \J_q}$ and  a.e. in $\mathbb R^N$,
        \item  $\nabla u_j\to\nabla u$ in $(L^{p}(K))^N$ for every $K\subset\subset \R^N$ and a.e. in $\mathbb R^N$,
        \item  $u_j\to u$ in $L^{p^*}(K)$ for every $K\subset\subset \R^N$ and a.e. in $\mathbb R^N$.
    \end{enumroman}
\end{corollary}
\begin{proof}
The first two items follow from Lemmas \ref{s0} and \ref{mbs1}, together
with a diagonal argument on an exhaustion of $\R^N\backslash\{x_{q,k}:k\in \J_q\}.$ Since $\J_q$ is finite, the excluded set has measure zero, which gives
the corresponding almost-everywhere convergence in $\R^N$.
The last two assertions follow from the local compact embeddings
\[
D^{2,q}(\R^N)\hookrightarrow\hookrightarrow W^{1,p}(K)
\qquad\text{and}\qquad
D^{2,q}(\R^N)\hookrightarrow\hookrightarrow L^{p^*}(K),
\]
which hold because $p<q^*$ and consequently $p^*<q^{**}$. 
\end{proof}

\section{The double critical case}\label{sec-prof11}

Solutions of problem \eqref{prob1} coincide with critical points of the $C^1$-functional $\Phi_\tau$ given in \eqref{energy-tau} with $\tau=1$. We will now investigate the compact properties of $\Phi_1$. Here we consider $1<p<q<p^*$ or $1<q<p<q^*$.

\begin{lemma}\label{lemma ps1}
Every $(PS)_c$ sequence of $\Phi_1$, with $c\in\mathbb R$, is bounded.
\end{lemma}

\begin{proof}
Let $\{u_j\}_{j\in\N}\subset \mathcal W$ be a $(PS)_c$ sequence for $\Phi_1$. Then
\begin{eqnarray}\label{ps1}
\int_{\R^N}\seq{\frac{1}{q}|\Delta u_j|^q+\frac{1}{p}|\nabla u_j|^p-\dfrac{\lambda}{r_\sigma}\dfrac{|u_j|^{r_\sigma}}{|x|^\sigma}-\dfrac{1}{q^{**}}|u_j|^{q^{**}}-\dfrac{1}{p^*}|u_j|^{p^*}}dx=c+o(1),
\end{eqnarray}
and
\begin{eqnarray}\label{ps2}
\int_{\R^N}\seq{|\Delta u_j|^q+|\nabla u_j|^p-\lambda\frac{|u_j|^{r_\sigma}}{|x|^\sigma}-|u_j|^{q^{**}}-|u_j|^{p^*}}dx=o(1)\|u_j\|.
\end{eqnarray}
Suppose by contradiction that $\|u_j\|\to+\infty$, and define
\[
t_j=I_s(u_j)^{-1/s},\quad \tilde{t}_j=t_j^{-1},\quad \w{u}_j=(u_j)_{t_j}.
\]
Since $I_s$ is coercive, $\tilde{t}_j\to+\infty$. Moreover,
\[
\w{u}_j\in\M,\quad u_j=(\w{u}_j)_{\tilde{t}_j},\quad\forall j\in\N.
\]
For $\ell>1$, set
\begin{equation}\label{defbetal}
\beta_\ell:=
\begin{cases}
-\gamma\ell+N,&1<p<q,\\
\gamma\ell-N,&1<q<p<q^*.
\end{cases}
\end{equation}
Thus, in either scaling regime,
\[
\int_{\R^N}|(\w{u}_j)_{\tilde{t}_j}|^\ell dx=\tilde{t}_j^{\beta_\ell}\int_{\R^N}|\w{u}_j|^\ell dx.
\]
Furthermore,
\[
\beta_{p^*}-s=\frac{(2q-p)p^*-pq}{|p-q|}>0,\qquad
\beta_{q^{**}}-s=\frac{(2q-p)q^{**}-pq}{|p-q|}>0.
\]
Indeed,
\[
(2q-p)p^*-pq=\frac{p\big(N(q-p)+pq\big)}{N-p},\qquad
(2q-p)q^{**}-pq=\frac{2q\big(N(q-p)+pq\big)}{N-2q},
\]
and $N(q-p)+pq>0$ if $p<q$, while, if $q<p$, this inequality is equivalent to $p<q^*$.
Using the scaling properties in \eqref{ps2}, we obtain
\begin{eqnarray}\label{nps1}
\tilde{t}_j^s\int_{\R^N}\left(|\Delta\w{u}_j|^q+|\nabla\w{u}_j|^p\right)dx
&=&\lambda\tilde{t}_j^s\int_{\R^N}\frac{|\w{u}_j|^{r_\sigma}}{|x|^\sigma}dx
+\tilde{t}_j^{\,\beta_{q^{**}}}\int_{\R^N}|\w{u}_j|^{q^{**}}dx\nonumber\\
&&+\tilde{t}_j^{\,\beta_{p^*}}\int_{\R^N}|\w{u}_j|^{p^*}dx+o(\tilde{t}_j^s).
\end{eqnarray}
Dividing \eqref{nps1} by $\tilde{t}_j^s$ gives
\begin{eqnarray}\label{nps2}
\int_{\R^N}\left(|\Delta\w{u}_j|^q+|\nabla\w{u}_j|^p\right)dx
&=&\lambda\int_{\R^N}\frac{|\w{u}_j|^{r_\sigma}}{|x|^\sigma}dx
+\tilde{t}_j^{\,\beta_{q^{**}}-s}\int_{\R^N}|\w{u}_j|^{q^{**}}dx\nonumber\\
&&+\tilde{t}_j^{\,\beta_{p^*}-s}\int_{\R^N}|\w{u}_j|^{p^*}dx+o(1).
\end{eqnarray}
Since $\{\w{u}_j\}_{j\in\N}\subset\M$ is bounded in $\mathcal{W}$ (and therefore in $L^{r_\sigma}_\sigma(\mathbb R^N)$), we get
\[
\int_{\R^N}|\w{u}_j|^{q^{**}}dx=O\left(\tilde{t}_j^{\,s-\beta_{q^{**}}}\right),\quad
\int_{\R^N}|\w{u}_j|^{p^*}dx=O\left(\tilde{t}_j^{\,s-\beta_{p^*}}\right).
\]
Since $\beta_{q^{**}}>s$, $\beta_{p^*}>s$, and $\tilde{t}_j\to+\infty$, it follows that
\begin{eqnarray}\label{nps22}
\int_{\R^N}|\w{u}_j|^{q^{**}}dx\to0,\quad
\int_{\R^N}|\w{u}_j|^{p^*}dx\to0.
\end{eqnarray}
Now consider a point $(r_1,\sigma_1)$ in an open neighborhood of $(r_\sigma,\sigma)$ such that $\mathcal W\hookrightarrow L_{\sigma_1}^{r_1}(\R^N)$ and $(r_\sigma,\sigma)$ lies on the open line segment connecting $(p^*,0)$ to $(r_1,\sigma_1)$. Then there exists $t\in(0,1)$ such that
\[
(r_\sigma,\sigma)=t(r_1,\sigma_1)+(1-t)(p^*,0),
\]
and Hölder's inequality gives
\begin{eqnarray}\label{nps222}
\int_{\R^N}\frac{|\w{u}_j|^{r_\sigma}}{|x|^\sigma}dx
&=&\int_{\R^N}\frac{|\w{u}_j|^{tr_1+(1-t)p^*}}{|x|^{t\sigma_1+(1-t)0}}dx
=\int_{\R^N}\frac{|\w{u}_j|^{tr_1}}{|x|^{t\sigma_1}}|\w{u}_j|^{(1-t)p^*}dx\nonumber\\
&\leq&\left(\int_{\R^N}\frac{|\w{u}_j|^{r_1}}{|x|^{\sigma_1}}dx\right)^t
\left(\int_{\R^N}|\w{u}_j|^{p^*}dx\right)^{1-t}.
\end{eqnarray}
Since $\{\w{u}_j\}_{j\in\N}$ is bounded in $L_{\sigma_1}^{r_1}(\R^N)$, \eqref{nps22} and \eqref{nps222} imply
\begin{eqnarray}\label{nps2222}
\int_{\R^N}\frac{|\w{u}_j|^{r_\sigma}}{|x|^\sigma}\to0.
\end{eqnarray}
Applying the same scaling to \eqref{ps1} and dividing by $\tilde{t}_j^s$, we obtain
\begin{eqnarray}\label{nps3}
\int_{\R^N}\left(\frac{1}{q}|\Delta\w{u}_j|^q+\frac{1}{p}|\nabla\w{u}_j|^p\right)dx
&=&\frac{\lambda}{r_\sigma}\int_{\R^N}\frac{|\w{u}_j|^{r_\sigma}}{|x|^\sigma}dx
+\frac{\tilde{t}_j^{\,\beta_{q^{**}}-s}}{q^{**}}\int_{\R^N}|\w{u}_j|^{q^{**}}dx\nonumber\\
&&+\frac{\tilde{t}_j^{\,\beta_{p^*}-s}}{p^*}\int_{\R^N}|\w{u}_j|^{p^*}dx+o(1)\nonumber\\
&\leq&\frac{\lambda}{r_\sigma}\int_{\R^N}\frac{|\w{u}_j|^{r_\sigma}}{|x|^\sigma}dx
+\frac{\tilde{t}_j^{\,\beta_{q^{**}}-s}}{p^*}\int_{\R^N}|\w{u}_j|^{q^{**}}dx\nonumber\\
&&+\frac{\tilde{t}_j^{\,\beta_{p^*}-s}}{p^*}\int_{\R^N}|\w{u}_j|^{p^*}dx+o(1),
\end{eqnarray}
where we used $p^*<q^{**}$, which follows from $p<q^*$. Multiplying \eqref{nps3} by $p^*$ yields
\begin{eqnarray}\label{nps4}
\int_{\R^N}\left(\frac{p^*}{q}|\Delta\w{u}_j|^q+\frac{p^*}{p}|\nabla\w{u}_j|^p\right)dx
&\leq&\frac{\lambda p^*}{r_\sigma}\int_{\R^N}\frac{|\w{u}_j|^{r_\sigma}}{|x|^\sigma}dx
+\tilde{t}_j^{\,\beta_{q^{**}}-s}\int_{\R^N}|\w{u}_j|^{q^{**}}dx\nonumber\\
&&+\tilde{t}_j^{\,\beta_{p^*}-s}\int_{\R^N}|\w{u}_j|^{p^*}dx+o(1).
\end{eqnarray}
Subtracting \eqref{nps2} from \eqref{nps4}, we obtain
\begin{eqnarray}\label{nps5}
\int_{\R^N}\left(\left(\frac{p^*}{q}-1\right)|\Delta\w{u}_j|^q+\left(\frac{p^*}{p}-1\right)|\nabla\w{u}_j|^p\right)dx
\leq\lambda\left|\frac{p^*}{r_\sigma}-1\right|\int_{\R^N}\frac{|\w{u}_j|^{r_\sigma}}{|x|^\sigma}dx+o(1).
\end{eqnarray}
In both regimes considered here, we have $q<p^*$: this is assumed when $p<q<p^*$ and follows from $q<p<p^*$ when $q<p<q^*$. Therefore,
\[
\frac{p^*}{q}-1>0,\qquad \frac{p^*}{p}-1>0.
\]
Combining \eqref{nps2222} and \eqref{nps5}, we conclude that
\[
\left(\frac{p^*}{q}-1\right)\int_{\R^N}|\Delta\w{u}_j|^qdx+
\left(\frac{p^*}{p}-1\right)\int_{\R^N}|\nabla\w{u}_j|^pdx\leq o(1),
\]
which contradicts $\w{u}_j\in\M$. Hence $\{u_j\}_{j\in\N}$ is bounded.
\end{proof}

For the next Lemma, let us recall the definition of $S_{1,p}$ and $S_{2,q}$ in \eqref{S1pS2q}. The following characterization of the threshold \(c^*\) is inspired by \cite{MR4846130} (see also \cite{MR4349776}). Let $\mathcal{X}$ be the set of all nonzero points $(a,b,c,d)\in\R^4$ such that 
\begin{equation}\label{defXfull}
a,b,c,d\geq0,\quad a+b=c+d,\quad c\leq S_{2,q}^{-q^{**}/q}a^{q^{**}/q},\quad d\leq S_{1,p}^{-p^*/p}b^{p^{*}/p}.
\end{equation}
Let $h$ be defined as
\begin{equation}\label{defhc*}
h(a,b,c,d)=\frac{1}{q}a+\frac{1}{p}b-\frac{1}{q^{**}}c-\frac{1}{p^*}d.
\end{equation}
Finally, define
\begin{eqnarray}\label{c-star}
    c^*=\inf h(\mathcal{X}).
\end{eqnarray}

\begin{lemma}
  Assuming $1<p<q< p^*$ or $1<q<p<q^*$  we have that $c^*>0$.
\end{lemma}
\begin{proof}

    Let $(a,b,c,d)\in\mathcal{X}$. 
     Since $0<a+b=c+d$, $c\leq S_{2,q}^{-q^{**}/q}a^{q^{**}/q}$ and $d\leq S_{1,p}^{-p^*/p}b^{p^{*}/p}$, one gets
    \begin{eqnarray*}
        a+b&\leq& S_{2,q}^{-q^{**}/q}a^{q^{**}/q}+S_{1,p}^{-p^*/p}b^{p^{*}/p}
        \\
        &\leq& \max\set{S_{2,q}^{-q^{**}/q},S_{1,p}^{-p^*/p}}\left[(a+b)^{q^{**}/q}+(a+b)^{p^*/p}\right]
    \end{eqnarray*}
    which gives
    \begin{eqnarray*}
        \dfrac{1}{\max\set{S_{2,q}^{-q^{**}/q},S_{1,p}^{-p^*/p}}}\leq (a+b)^{(q^{**}/q)-1}+(a+b)^{(p^*/p)-1},
    \end{eqnarray*}
    and this shows that there exists a  constant $c_1>0$ such that 
    \begin{equation}\label{inf a+b}
    a+b>c_1,\quad\forall (a,b)\in \pi_{a,b}(\mathcal{X}),
    \end{equation}
    where $\pi_{a,b}$ is the projection onto the first two coordinates.
Now, writing  $d=a+b-c$, we get
    \begin{eqnarray*}
         h(a,b,c,d)&=&\seq{\frac{1}{q}-\frac{1}{p^*}}a+\seq{\frac{1}{p}-\frac{1}{p^*}}b+\seq{\frac{1}{p^*}-\frac{1}{q^{**}}}c.
    \end{eqnarray*}
In either case $1<p<q<p^*$ or $1<q<p<q^*$ we get $p<q^*$ and $q<p^*$.  Then
       \begin{eqnarray*}
         h(a,b,c,d) &\geq&\min\set{\seq{\dfrac{1}{q}-\dfrac{1}{p^*}},\seq{\frac{1}{p}-\frac{1}{p^*}}}(a+b).
    \end{eqnarray*}
    This together with \eqref{inf a+b} give us a constant $c_2>0$ such that
    \begin{eqnarray*}
        h(a,b,c,d)>c_2,\quad\forall (a,b,c,d)\in\mathcal{X}.
    \end{eqnarray*}
    This concludes the proof.
\end{proof}

We are now ready to prove the $(PS)_c$ condition for $\Phi_1$.

\begin{lemma}\label{lemma ps full}
The functional $\Phi_1$ defined in
\eqref{energy-tau} satisfies the $(PS)_c$ condition for every $c<c^*$, where $c^*$ is defined in \eqref{c-star}.
\end{lemma}

\begin{proof}
Let $\{u_j\}_{j\in\N}\subset \mathcal W$ be a $(PS)_c$ sequence for $\Phi_1$.  We 
know from Lemma \ref{lemma ps1} that
$\{u_j\}_{j\in\N}$ is bounded in $\mathcal W$.  Passing to a subsequence, there exists
$u\in \mathcal W$ such that $u_j\rightharpoonup u$ in $\mathcal W.$
By Corollary \ref{cor compact emb}, after passing to a further subsequence,
\begin{equation}\label{weighted-convergence-full-PS}
 u_j\to u\quad\text{in }L^{r_\sigma}_\sigma(\R^N).
\end{equation}
Corollary \ref{s2} allows us to choose this
subsequence so that
\begin{equation}\label{pointwise-convergences-full-PS}
\begin{aligned}
 u_j&\to u &&\text{a.e. in }\R^N,\\
 \nabla u_j&\to\nabla u &&\text{a.e. in }\R^N,\\
 \Delta u_j&\to\Delta u &&\text{a.e. in }\R^N.
\end{aligned}
\end{equation}
We now show that $u$ is a critical point of $\Phi_1$. Let
$\phi\in C_c^\infty(\R^N)$. Since $\{\Delta u_j\}$ is bounded in
$L^q(\R^N)$ and $\Delta u_j\to\Delta u$ a.e. in $\R^N$, we have
\[
|\Delta u_j|^{q-2}\Delta u_j
\rightharpoonup
|\Delta u|^{q-2}\Delta u
\qquad\text{in }L^{q'}(\R^N).
\]
Likewise,
\[
|\nabla u_j|^{p-2}\nabla u_j
\rightharpoonup
|\nabla u|^{p-2}\nabla u
\qquad\text{in }(L^{p'}(\R^N))^N.
\]
Moreover, by \eqref{pointwise-convergences-full-PS} and the boundedness of
$\{u_j\}$ in $L^{p^*}(\R^N)\cap L^{q^{**}}(\R^N)$,
\[
|u_j|^{p^*-2}u_j
\rightharpoonup
|u|^{p^*-2}u
\quad\text{in }L^{(p^*)'}(\R^N)
\quad\text{and}\quad
|u_j|^{q^{**}-2}u_j
\rightharpoonup
|u|^{q^{**}-2}u
\quad\text{in }L^{(q^{**})'}(\R^N).
\]
Finally, \eqref{weighted-convergence-full-PS} gives
\[
\int_{\R^N}
\frac{|u_j|^{r_\sigma-2}u_j\phi}{|x|^\sigma}\,dx
\rightarrow
\int_{\R^N}
\frac{|u|^{r_\sigma-2}u\phi}{|x|^\sigma}\,dx.
\]
Passing to the limit in $\langle E^{'}_1(u_j),\phi\rangle=o(1),$ we obtain
\[
\langle E^{'}_1(u),\phi\rangle=0
\qquad\text{for every }\phi\in C_c^\infty(\R^N).
\]
Since $C_c^\infty(\R^N)$ is dense in $\mathcal W$, $u$ is a weak solution of \eqref{prob1}.

Set $ v_j=u_j-u.$ Since $\{v_j\}_{j\in\N}$ is bounded in $\mathcal W$, it is bounded in
$D^{2,q}(\R^N)$, $D^{1,p}(\R^N)$, $L^{q^{**}}(\R^N)$, and
$L^{p^*}(\R^N)$.  In view of the three almost-everywhere convergences in
\eqref{pointwise-convergences-full-PS}, the Brezis--Lieb lemma gives
\begin{align}
 \int_{\R^N}|\Delta u_j|^q\,dx
 &=\int_{\R^N}|\Delta u|^q\,dx
  +\int_{\R^N}|\Delta v_j|^q\,dx+o(1),
 \label{BL-Delta-full}\\
 \int_{\R^N}|\nabla u_j|^p\,dx
 &=\int_{\R^N}|\nabla u|^p\,dx
  +\int_{\R^N}|\nabla v_j|^p\,dx+o(1),
 \label{BL-gradient-full}\\
 \int_{\R^N}|u_j|^{q^{**}}\,dx
 &=\int_{\R^N}|u|^{q^{**}}\,dx
  +\int_{\R^N}|v_j|^{q^{**}}\,dx+o(1),
 \label{BL-q-critical-full}\\
 \int_{\R^N}|u_j|^{p^*}\,dx
 &=\int_{\R^N}|u|^{p^*}\,dx
  +\int_{\R^N}|v_j|^{p^*}\,dx+o(1).
 \label{BL-p-critical-full}
\end{align}
Moreover, \eqref{weighted-convergence-full-PS} implies
\begin{equation}\label{weighted-energy-convergence-full}
 \int_{\R^N}\frac{|u_j|^{r_\sigma}}{|x|^\sigma}\,dx
 =\int_{\R^N}\frac{|u|^{r_\sigma}}{|x|^\sigma}\,dx+o(1).
\end{equation}
The four-dimensional sequence
\[
 \left(
 \int_{\R^N}|\Delta v_j|^q\,dx,
 \int_{\R^N}|\nabla v_j|^p\,dx,
 \int_{\R^N}|v_j|^{q^{**}}\,dx,
 \int_{\R^N}|v_j|^{p^*}\,dx
 \right)
\]
is bounded in $[0,\infty)^4$.  Thus, after passing to one further common
subsequence, there are $a,b,c_0,d\geq0$ such that
\begin{align}
 \int_{\R^N}|\Delta v_j|^q\,dx&\to a,&
 \int_{\R^N}|\nabla v_j|^p\,dx&\to b,
 \label{residual-limits-full-1}\\
 \int_{\R^N}|v_j|^{q^{**}}\,dx&\to c_0,&
 \int_{\R^N}|v_j|^{p^*}\,dx&\to d.
 \label{residual-limits-full-2}
\end{align}
Since $E^{'}_1(u_j)u_j=o(1)$, $E^{'}_1(u)u=0$, and the weighted term converges by
\eqref{weighted-energy-convergence-full}, subtracting these two identities
and using \eqref{BL-Delta-full}--\eqref{BL-p-critical-full} yields
\begin{equation}\label{residual-balance-full-PS}
 a+b=c_0+d.
\end{equation}
The Sobolev inequalities and \eqref{residual-limits-full-1}--
\eqref{residual-limits-full-2} give
\begin{equation}\label{residual-Sobolev-full-PS}
 c_0\leq S_{2,q}^{-q^{**}/q}a^{q^{**}/q},
 \qquad
 d\leq S_{1,p}^{-p^*/p}b^{p^*/p}.
\end{equation}
The energy identities and the same splitting relations give
\begin{equation}\label{energy-splitting-full-PS}
 c=\Phi_1(u)+h(a,b,c_0,d),
\end{equation}
where $h$ is given in \eqref{defhc*}. We next show that the energy of the weak limit is nonnegative.  Testing
$\Phi_1'(u)=0$ with $u$ gives
\begin{equation}\label{Nehari-limit-full-PS}
 \int_{\R^N}|\Delta u|^q\,dx
 +\int_{\R^N}|\nabla u|^p\,dx
 =\lambda\int_{\R^N}\frac{|u|^{r_\sigma}}{|x|^\sigma}\,dx
 +\int_{\R^N}|u|^{q^{**}}\,dx
 +\int_{\R^N}|u|^{p^*}\,dx.
\end{equation}
On the other hand, we notice that $u$ has the regularity required to invoke Proposition \ref{prop:pohozaev-identity}, with $\tau_1=\tau_2=1$. Indeed, $u$ is a weak solution of \eqref{generalfproblem}, with $f(x,u)$ as in the right-hand side of \eqref{prob1}, which satisfies $\eqref{condf_0}$ with $\Omega=\mathbb R^N\backslash\{0\}$. This implies that we can invoke Theorem~\ref{thm:intro-regularity} and the Pohozaev identity yields
\begin{align}
 &\frac{N-2q}{q}\int_{\R^N}|\Delta u|^q\,dx
 +\frac{N-p}{p}\int_{\R^N}|\nabla u|^p\,dx
 \notag\\&
 \qquad=\lambda\frac{N-\sigma}{r_\sigma}
   \int_{\R^N}\frac{|u|^{r_\sigma}}{|x|^\sigma}\,dx\quad+\frac{N}{q^{**}}\int_{\R^N}|u|^{q^{**}}\,dx
 +\frac{N}{p^*}\int_{\R^N}|u|^{p^*}\,dx.
 \label{Pohozaev-limit-full-PS}
\end{align}
Because
\[
 \frac{N}{q^{**}}=\frac{N-2q}{q},
 \qquad
 \frac{N}{p^*}=\frac{N-p}{p},
\]
subtracting the critical terms from the corresponding principal terms in
\eqref{Nehari-limit-full-PS}--\eqref{Pohozaev-limit-full-PS}, and using $ r_\sigma(2q-p)=pq+\sigma(q-p),$ gives
\begin{align}
 &\left(\frac1q-\frac1{r_\sigma}\right)
 \left(\int_{\R^N}|\Delta u|^q\,dx
       -\int_{\R^N}|u|^{q^{**}}\,dx\right)
 \notag\\
 &\qquad+
 \left(\frac1p-\frac1{r_\sigma}\right)
 \left(\int_{\R^N}|\nabla u|^p\,dx
       -\int_{\R^N}|u|^{p^*}\,dx\right)=0.
 \label{scaled-cancellation-full-PS}
\end{align}
Combining \eqref{Nehari-limit-full-PS} and
\eqref{scaled-cancellation-full-PS} with the definition of $\Phi_1$, we obtain
\begin{equation}\label{energy-limit-full-PS}
 \Phi_1(u)=\frac2N\int_{\R^N}|u|^{q^{**}}\,dx
      +\frac1N\int_{\R^N}|u|^{p^*}\,dx\geq0.
\end{equation}

\noindent Suppose that $(a,b,c_0,d)\neq(0,0,0,0)$.  Then
\eqref{residual-balance-full-PS} and \eqref{residual-Sobolev-full-PS}
show that $(a,b,c_0,d)\in\mathcal X$, where $\mathcal X$ is defined in \eqref{defXfull}.  Consequently, $h(a,b,c_0,d)\geq c^*.$ Together with \eqref{energy-splitting-full-PS} and
\eqref{energy-limit-full-PS}, this gives $c\geq c^*$, a contradiction.
Therefore  $a=b=c_0=d=0.$ In particular, \eqref{residual-limits-full-1} gives
\[
 \int_{\R^N}|\Delta(u_j-u)|^q\,dx\to0,
 \qquad
 \int_{\R^N}|\nabla(u_j-u)|^p\,dx\to0.
\]
Hence $u_j\to u$ strongly in $\mathcal W$, proving the $(PS)_c$ condition.
\end{proof}

We are now ready to conclude the proof of Theorem \ref{intro:thm:p-less-q}.

\subsection*{Proof of Theorem \ref{intro:thm:p-less-q}}

Through this section, for $\ell>1$ we consider $\beta_\ell$ as in \eqref{defbetal}, so that
\[
\int_{\R^N}|u_t|^\ell dx=t^{\beta_\ell}\int_{\R^N}|u|^\ell dx.
\]
As shown in the proof of Lemma \ref{lemma ps1}, $\beta_{q^{**}}>s$, and $\beta_{p^*}>s.$ Moreover, a direct computation gives
\[
\beta_{q^{**}}=\frac{N}{N-2q}s.
\]

\begin{proof}[Proof of Theorem \ref{intro:thm:p-less-q}]
Fix $k,m\geq1$ such that
\[
\lambda_k=\cdots=\lambda_{k+m-1}<\lambda_{k+m}.
\]
We will employ Theorem \ref{new m1} with $\Phi:=\Phi_1$ and $c^*$ given by \eqref{c-star}. Let $n$ be large so that
\[
\frac1n\in(0,\lambda_{k+m}-\lambda_{k+m-1}).
\]
Then $i(\M\backslash\w{\Psi}_{\lambda_{k+m-1}+1/n})=k+m-1$ by Theorem \ref{Theorem 301} \ref{301-3}. Since $\M\backslash\w{\Psi}_{\lambda_{k+m-1}+1/n}$ is an open symmetric set of index $k+m-1$, it contains a compact symmetric subset $C$ with $i(C)=k+m-1$ (see the proof of Proposition 3.1 in Degiovanni and Lancelotti \cite{MR2371112}). We apply Theorem \ref{new m1} with
\[
A_0=A_0(n)=C,\qquad B_0=\w{\Psi}_{\lambda_k}.
\]
If $\lambda_1=\cdots=\lambda_k$, then
$B_0=\widetilde{\Psi}_{\lambda_k}=\M$. Hence
$M_s\setminus B_0=\varnothing$, and therefore
\[
i(M_s\setminus B_0)=0\leq k-1.
\] If $\lambda_{l-1}<\lambda_l=\lambda_k=\cdots=\lambda_{k+m-1}$ for some $2\leq l\leq k$, then
\[
i(\M\backslash B_0)=i(\M\backslash\w{\Psi}_{\lambda_l})=l-1\leq k-1
\]
by Theorem \ref{Theorem 301} \ref{301-3}.
Let $R>\rho>0$ and set
\begin{eqnarray*}
X&=&\set{u_t:u\in A_0,\ 0\leq t\leq R},\quad
A=\set{u_R:u\in A_0}\quad\text{and}\quad
B=\set{u_\rho:u\in B_0}.
\end{eqnarray*}
For $u\in\M$ and $t\geq0$, we have
\begin{equation}\label{210-1}
\Phi_1(u_t)=t^s\seq{1-\frac{\lambda}{\w{\Psi}(u)}}-\frac{t^{\beta_{q^{**}}}}{q^{**}}\int_{\R^N}|u|^{q^{**}}dx-\frac{t^{\beta_{p^*}}}{p^*}\int_{\R^N}|u|^{p^*}dx.
\end{equation}
Since $\M $ is bounded and $\beta_{q^{**}},\beta_{p^*}>s$, we have
\[
\Phi_1(u_t)\geq t^s\seq{1-\frac{\lambda}{\w{\Psi}(u)}+o(1)}\quad\text{as }t\to0
\]
uniformly on $B_0$. Since $\w{\Psi}(u)\geq\lambda_k$ for $u\in B_0$ and $\lambda<\lambda_k$, there exists $\rho>0$ sufficiently small such that
\[
\inf_{u\in B}\Phi_1(u)>0.
\]

\noindent For every $u\in A_0\subset\M \backslash\w{\Psi}_{\lambda_{k+m-1}+1/n}$, we have
\[
\frac{1}{r_\sigma}\int_{\R^N}\frac{|u|^{r_\sigma}}{|x|^\sigma}dx>\frac{1}{\lambda_{k+m-1}+1/n}>\frac{1}{\lambda_{k+m}},
\]
and hence
\begin{equation}\label{210-11}
\int_{\R^N}\frac{|u|^{r_\sigma}}{|x|^\sigma}dx>\frac{r_\sigma}{\lambda_{k+m}}.
\end{equation}
We claim that there exists $C>0$, independent of $n$, such that
\begin{equation}\label{uniform-qcritical-lower-bound}
\int_{\R^N}|u|^{q^{**}}dx\geq C\qquad\text{for every }u\in A_0(n).
\end{equation}
Suppose otherwise. Then, after passing to a sequence  $n\to\infty$, there exist $u_n\in A_0(n)\subset\M $ such that $u_n\to 0$ in $L^{q^{**}}(\R^N)$.
Since $\M $ is bounded in $\mathcal W$, after passing to a subsequence we have $u_n\rightharpoonup u$ in $\mathcal W$. The preceding convergence implies $u=0$. By the compact embedding $\mathcal W\hookrightarrow L^{r_\sigma}_\sigma(\R^N)$, we obtain
\[
\int_{\R^N}\frac{|u_n|^{r_\sigma}}{|x|^\sigma}dx\to0,
\]
contradicting \eqref{210-11}. Thus \eqref{uniform-qcritical-lower-bound} holds.
Since $\lambda_{k+m-1}=\lambda_k$, equations \eqref{210-1} and \eqref{uniform-qcritical-lower-bound} give
\[
\Phi_1(u_t)\leq t^s\seq{1-\frac{\lambda}{\lambda_k+1/n}}-c_1t^{\beta_{q^{**}}}
=t^s\seq{1-\frac{\lambda}{\lambda_k+1/n}}-c_1t^{\frac{N}{N-2q}s}
\]
for all $u\in A_0$ and some constant $c_1>0$ independent of $n$. Consequently,
\[
\sup_{u\in A}\Phi_1(u)\leq R^s-c_1R^{\frac{N}{N-2q}s}\leq0
\]
if $R$ is sufficiently large. This proves the first inequality in \eqref{new m1-2}. Moreover,
\begin{align*}
\sup_{u\in X}\Phi_1(u)
&\leq\sup_{t\geq0}\set{t^s\seq{1-\frac{\lambda}{\lambda_k+1/n}}-c_1t^{\frac{N}{N-2q}s}}\\
&\leq\frac{2q}{N}\seq{\frac{N-2q}{Nc_1}}^{\frac{N-2q}{2q}}\seq{1-\frac{\lambda}{\lambda_k+1/n}}^{\frac{N}{2q}}.
\end{align*}
It follows that
\[
\sup_{u\in X}\Phi_1(u)<c^*
\]
whenever $\lambda<\lambda_k$ is sufficiently close to $\lambda_k$ and $n$ is sufficiently large. Thus the second inequality in \eqref{new m1-2} also holds.

By Lemma \ref{lemma ps full}, the functional $\Phi_1$ satisfies the $(PS)_c$ condition for every $c\in(0,c^*)$. Therefore, Theorem \ref{new m1} gives
\[
0<c_k^*\leq\cdots\leq c_{k+m-1}^*<c^*,
\]
where each $c_j^*$ is a critical value of $\Phi_1$, and $\Phi_1$ possesses $m$ distinct pairs of associated critical points. Since all these critical values are positive, the corresponding critical points are nontrivial. Consequently, there exists $\delta_k>0$ such that, for every $\lambda\in(\lambda_k-\delta_k,\lambda_k),$ problem \eqref{prob1} possesses $m$ distinct pairs of nontrivial weak solutions at positive energy levels.
\end{proof}

 \section{One critical term}\label{sec-proof12}

 In this regime we consider problem \eqref{intro:problem-p-less-q-large}, whose
solutions are the critical points of $\Phi_0$, which is given in \eqref{energy-tau} for $\tau=0$. Here we consider $p^*<q$.

The $p^*$-critical term is absent, so we remove the fourth residual variable.
Let $\mathcal X_{\rm red}$ be the set of all nonzero triples
$(a,b,c)\in[0,\infty)^3$ satisfying
\begin{equation}\label{residual-set-reduced}
 a+b=c,\qquad
 c\leq S_{2,q}^{-q^{**}/q}a^{q^{**}/q},
\end{equation}
and set
\[
 h(a,b,c)=\frac1q a+\frac1p b-\frac1{q^{**}}c,
 \qquad
 c^*_{\rm red}=\inf_{(a,b,c)\in\mathcal X_{\rm red}}h(a,b,c).
\]

\begin{lemma}\label{c-star-reduced}
If $p^*<q$, then
\begin{equation}\label{exact-c-star-reduced}
 c^*_{\rm red}=\frac2N S_{2,q}^{N/(2q)}>0.
\end{equation}
\end{lemma}

\begin{proof}
Since $p^*<q$, we have $p<q<q^{**}$.  For every
$(a,b,c)\in\mathcal X_{\rm red}$,
\begin{align*}
 h(a,b,c)
 &=\left(\frac1q-\frac1{q^{**}}\right)a
  +\left(\frac1p-\frac1{q^{**}}\right)b
 \geq\frac2N(a+b)=\frac2N c.
\end{align*}
Furthermore, $a\leq c$ and \eqref{residual-set-reduced} imply $c\leq S_{2,q}^{-q^{**}/q}c^{q^{**}/q}.$ Since $c>0$ (if $c=0$, then $a+b=0$ and so the triple $(a,b,c)$ is trivial), it follows that
\[
 c\geq S_{2,q}^{q^{**}/(q^{**}-q)}=S_{2,q}^{N/(2q)}.
\]
Thus $h(a,b,c)\geq(2/N)S_{2,q}^{N/(2q)}$.  Equality is achieved in the
finite-dimensional constraint set at $ a=c=S_{2,q}^{N/(2q)}, b=0.$ This proves \eqref{exact-c-star-reduced}.
\end{proof}

\begin{lemma}\label{lemma ps0 bounded}
Assume $p^{*}<q$. Then every $(PS)_c$ sequence of $\Phi_0$, with $c\in\mathbb{R}$,
is bounded in $\mathcal W$.
\end{lemma}

\begin{proof}
Let $\{u_j\}_{j\in\mathbb{N}}\subset \mathcal W$ be a $(PS)_c$ sequence for $\Phi_0$. Then
\begin{equation}\label{ps1red}
\int_{\mathbb{R}^N}\left(\frac1q|\Delta u_j|^{q}+\frac1p|\nabla u_j|^{p}
-\frac{\lambda}{r_\sigma}\frac{|u_j|^{r_\sigma}}{|x|^{\sigma}}
-\frac{1}{q^{**}}|u_j|^{q^{**}}\right)dx=c+o(1),
\end{equation}
and
\begin{equation}\label{ps2red}
\int_{\mathbb{R}^N}\left(|\Delta u_j|^{q}+|\nabla u_j|^{p}
-\lambda\frac{|u_j|^{r_\sigma}}{|x|^{\sigma}}-|u_j|^{q^{**}}\right)dx=o(1)\|u_j\|.
\end{equation}
Suppose by contradiction that $\|u_j\|\to+\infty$, and define
$$
t_j=I_s(u_j)^{-1/s},\qquad \tilde t_j=t_j^{-1},\qquad
\widetilde u_j=(u_j)_{t_j}.
$$
Since $I_s$ is coercive, $\tilde t_j\to+\infty$. Moreover,
$$
\widetilde u_j\in \mathcal{M},\qquad u_j=(\widetilde u_j)_{\tilde t_j},\qquad\forall j\in\mathbb{N}.
$$
Since $p<p^{*}<q$, we are in the regime $1<p<q$, so that $\beta_\ell=-\gamma\ell+N$ for
$\ell>1$, as in \eqref{defbetal}, and
$$
\int_{\mathbb{R}^N}|(\widetilde u_j)_{\tilde t_j}|^{\ell}\,dx
=\tilde t_j^{\,\beta_\ell}\int_{\mathbb{R}^N}|\widetilde u_j|^{\ell}\,dx .
$$
As in the proof of Lemma \ref{lemma ps1},
$$
\beta_{q^{**}}-s=\frac{(2q-p)q^{**}-pq}{q-p}>0,
\quad\text{since}\quad
(2q-p)q^{**}-pq=\frac{2q\big(N(q-p)+pq\big)}{N-2q}>0 .
$$
Also, by \eqref{ut bounded},
$\|u_j\|\le\max\{\tilde t_j^{\,s/q},\tilde t_j^{\,s/p}\}\,\|\widetilde u_j\|$, and since
$p,q>1$ and $\{\widetilde u_j\}\subset \mathcal{M}$ is bounded in $\mathcal W$, we have
$o(1)\|u_j\|=o(\tilde t_j^{\,s})$.
Using the scaling properties in \eqref{ps2red}, we obtain
\begin{equation}\label{nps1red}
\tilde t_j^{\,s}\int_{\mathbb{R}^N}\big(|\Delta\widetilde u_j|^{q}+|\nabla\widetilde u_j|^{p}\big)dx
=\lambda\,\tilde t_j^{\,s}\int_{\mathbb{R}^N}\frac{|\widetilde u_j|^{r_\sigma}}{|x|^{\sigma}}dx
+\tilde t_j^{\,\beta_{q^{**}}}\int_{\mathbb{R}^N}|\widetilde u_j|^{q^{**}}dx
+o(\tilde t_j^{\,s}).
\end{equation}
Dividing \eqref{nps1red} by $\tilde t_j^{\,s}$ gives
\begin{equation}\label{nps2red}
\int_{\mathbb{R}^N}\big(|\Delta\widetilde u_j|^{q}+|\nabla\widetilde u_j|^{p}\big)dx
=\lambda\int_{\mathbb{R}^N}\frac{|\widetilde u_j|^{r_\sigma}}{|x|^{\sigma}}dx
+\tilde t_j^{\,\beta_{q^{**}}-s}\int_{\mathbb{R}^N}|\widetilde u_j|^{q^{**}}dx+o(1).
\end{equation}
Since $\{\widetilde u_j\}_{j\in\mathbb{N}}\subset \mathcal{M}$ is bounded in $\mathcal W$ and $L^{r_\sigma
}_\sigma(\R^N)$, we get
$$
\int_{\mathbb{R}^N}|\widetilde u_j|^{q^{**}}dx=O\Big(\tilde t_j^{\,s-\beta_{q^{**}}}\Big),
$$
and, since $\beta_{q^{**}}>s$ and $\tilde t_j\to+\infty$,
\begin{equation}\label{nps22red}
\int_{\mathbb{R}^N}|\widetilde u_j|^{q^{**}}\to 0 .
\end{equation}
Now we interpolate the weighted term between $L^{q^{**}}(\mathbb{R}^N)$ and a nearby weighted
space. For $\mu>1$, set
$$
(r_1,\sigma_1):=(q^{**},0)+\mu\big((r_\sigma,\sigma)-(q^{**},0)\big)
=\big(q^{**}+\mu(r_\sigma-q^{**}),\,\mu\sigma\big).
$$
By Corollary \ref{cor compact emb} we have $1\le r_\sigma\in(p^{*}_{\sigma},q^{**}_{\sigma})$, and
$\sigma\in(p,2q)\subset[0,\overline\eta)$, since $q<N/2$. These conditions define an open set of
parameters, so for $\mu>1$ sufficiently close to $1$ we still have $\sigma_1\in[0,\overline\eta)$
and $1\le r_1\in(p^{*}_{\sigma_1},q^{**}_{\sigma_1})$; hence $\mathcal W\hookrightarrow
L^{r_1}_{\sigma_1}(\mathbb{R}^N)$ by Lemma \ref{lem:weighted-embedding-E}. Taking
$t:=1/\mu\in(0,1)$ we get
$$
(r_\sigma,\sigma)=t(r_1,\sigma_1)+(1-t)(q^{**},0),
$$
and H\"older's inequality gives
\begin{equation}\label{nps222red}
\int_{\mathbb{R}^N}\frac{|\widetilde u_j|^{r_\sigma}}{|x|^{\sigma}}dx
=\int_{\mathbb{R}^N}\frac{|\widetilde u_j|^{tr_1}}{|x|^{t\sigma_1}}\,|\widetilde u_j|^{(1-t)q^{**}}dx
\le\left(\int_{\mathbb{R}^N}\frac{|\widetilde u_j|^{r_1}}{|x|^{\sigma_1}}dx\right)^{t}
\left(\int_{\mathbb{R}^N}|\widetilde u_j|^{q^{**}}dx\right)^{1-t}.
\end{equation}
Since $\{\widetilde u_j\}_{j\in\mathbb{N}}$ is bounded in $L^{r_1}_{\sigma_1}(\mathbb{R}^N)$,
\eqref{nps22red} and \eqref{nps222red} imply
\begin{equation}\label{nps2222red}
\int_{\mathbb{R}^N}\frac{|\widetilde u_j|^{r_\sigma}}{|x|^{\sigma}}dx\to 0 .
\end{equation}
Applying the same scaling to \eqref{ps1red} and dividing by $\tilde t_j^{\,s}$ (note that
$c\,\tilde t_j^{\,-s}=o(1)$, since $s>0$), we obtain
$$
\int_{\mathbb{R}^N}\left(\frac1q|\Delta\widetilde u_j|^{q}+\frac1p|\nabla\widetilde u_j|^{p}\right)dx
=\frac{\lambda}{r_\sigma}\int_{\mathbb{R}^N}\frac{|\widetilde u_j|^{r_\sigma}}{|x|^{\sigma}}dx
+\frac{\tilde t_j^{\,\beta_{q^{**}}-s}}{q^{**}}\int_{\mathbb{R}^N}|\widetilde u_j|^{q^{**}}dx+o(1).
$$
Multiplying by $q^{**}$ yields
\begin{equation}\label{nps4red}
\int_{\mathbb{R}^N}\left(\frac{q^{**}}{q}|\Delta\widetilde u_j|^{q}
+\frac{q^{**}}{p}|\nabla\widetilde u_j|^{p}\right)dx
=\frac{\lambda q^{**}}{r_\sigma}\int_{\mathbb{R}^N}\frac{|\widetilde u_j|^{r_\sigma}}{|x|^{\sigma}}dx
+\tilde t_j^{\,\beta_{q^{**}}-s}\int_{\mathbb{R}^N}|\widetilde u_j|^{q^{**}}dx+o(1).
\end{equation}
Subtracting \eqref{nps2red} from \eqref{nps4red}, the critical term cancels exactly and we get
\begin{equation}\label{nps5red}
\int_{\mathbb{R}^N}\left[\left(\frac{q^{**}}{q}-1\right)|\Delta\widetilde u_j|^{q}
+\left(\frac{q^{**}}{p}-1\right)|\nabla\widetilde u_j|^{p}\right]dx
=\lambda\left(\frac{q^{**}}{r_\sigma}-1\right)
\int_{\mathbb{R}^N}\frac{|\widetilde u_j|^{r_\sigma}}{|x|^{\sigma}}dx+o(1).
\end{equation}
In the present regime $p<q<q^{**}$, so
$$
\frac{q^{**}}{q}-1>0,\qquad \frac{q^{**}}{p}-1>0 .
$$
Combining \eqref{nps2222red} and \eqref{nps5red}, we conclude that
$$
\left(\frac{q^{**}}{q}-1\right)\int_{\mathbb{R}^N}|\Delta\widetilde u_j|^{q}
+\left(\frac{q^{**}}{p}-1\right)\int_{\mathbb{R}^N}|\nabla\widetilde u_j|^{p}\le o(1),
$$
which contradicts the fact that $\widetilde u_j\in \mathcal{M}$. Hence $\{u_j\}_{j\in\mathbb{N}}$ is bounded in $\mathcal W$.
\end{proof}

\begin{lemma}\label{lemma ps reduced}
Assume $p^*<q$. Then $\Phi_0$ satisfies the $(PS)_c$ condition for every
\[
 c<c^*_{\rm red}=\frac2N S_{2,q}^{N/(2q)}.
\]
\end{lemma}

\begin{proof}
By Lemma \ref{lemma ps0 bounded} we have that the $(PS)_c$ sequence is bounded. The proof now follows the same argument as Lemma~\ref{lemma ps full}, with
the $p^*$-term and its corresponding residual variable omitted. More
precisely, after passing to a subsequence, the Brezis--Lieb and derivative
splittings produce a triple $(a,b,c_0)\in[0,\infty)^3$ satisfying
\[
a+b=c_0,
\qquad
c_0\le S_{2,q}^{-q^{**}/q}a^{q^{**}/q},
\]
and
\[
c=\Phi_0(u)+h(a,b,c_0).
\]
By Theorem~\ref{thm:intro-regularity}, the weak limit \(u\) has the regularity required in Proposition~\ref{prop:pohozaev-identity}. Hence, taking
\(\tau_1=0\) and \(\tau_2=1\), the identities for \(u\) give
\[
\Phi_0(u)=\frac{2}{N}\int_{\R^N}|u|^{q^{**}}\,dx\ge0.
\]
If $(a,b,c_0)\neq(0,0,0)$, then $(a,b,c_0)\in \mathcal{X}_{\mathrm{red}}$, and hence
\[
c=\Phi_0(u)+h(a,b,c_0)\ge c_{\mathrm{red}}^*,
\]
a contradiction. Thus $a=b=c_0=0$, and the conclusion follows exactly
as in Lemma~\ref{lemma ps full}.
\end{proof}


\subsection*{Proof of Theorem \ref{intro:thm:p-less-q-large}}

The proof follows the same argument used in the proof of Theorem
\ref{intro:thm:p-less-q}, applied now to the functional $\Phi:=\Phi_0$ and with the compactness threshold given by Lemma \ref{c-star-reduced},
\(
c_{\mathrm{red}}^*.
\)

Indeed, since $p^*<q$, we have $p<q$, and hence the inverse scaling in
\eqref{ut definition} applies. For $u\in\mathcal{M}$ and $t\ge0$, we have
\[
\Phi_0(u_t)
=
t^s\left(1-\frac{\lambda}{\widetilde\Psi(u)}\right)
-\frac{t^{\beta_{q^{**}}}}{q^{**}}
\int_{\R^N}|u|^{q^{**}}\,dx,
\]
where $\beta_{q^{**}}>s.$
Thus, the construction of the sets $A_0$, $B_0$, $A$, $B$, and $X$ is
the same as in the proof of Theorem
\ref{intro:thm:p-less-q}.
Moreover, the argument leading to
\eqref{uniform-qcritical-lower-bound} remains unchanged and gives a
constant $C>0$, independent of the auxiliary integer $n$, such that
\[
\int_{\R^N}|u|^{q^{**}}\,dx\ge C
\qquad\text{for every }u\in A_0.
\]
Consequently, exactly as in the proof of Theorem
\ref{intro:thm:p-less-q}, we may choose $R> \rho>0$ and then take
$\lambda<\lambda_k$ sufficiently close to $\lambda_k$ so that
\[
\sup_{u\in A}\Phi_0(u)\le0
<
\inf_{u\in B}\Phi_0(u),
\qquad
\sup_{u\in X}\Phi_0(u)<c_{\mathrm{red}}^*.
\]
By Lemma \ref{lemma ps reduced}, $\Phi_0$ satisfies the
$(PS)_c$ condition for every $0<c<c_{\mathrm{red}}^*.$
Therefore, Theorem \ref{new m1} yields $m$ distinct pairs of nontrivial
critical points of $\Phi_0$ at positive energy levels. These critical
points are weak solutions of
\eqref{intro:problem-p-less-q-large}. Hence there exists $\delta_k>0$
such that, for every $\lambda\in(\lambda_k-\delta_k,\lambda_k),$ problem \eqref{intro:problem-p-less-q-large} possesses $m$ distinct
pairs of nontrivial weak solutions.

\section{Regularity results}\label{sec-regularity}

We begin by rewriting the fourth-order equation as an equivalent
coupled second-order system. This formulation separates the relation
between $u$ and the nonlinear flux $w:=|\Delta u|^{q-2}\Delta u$ from the
equation satisfied by the flux, and will be the starting point for the
regularity arguments developed below. Let
\[\mathfrak a_r(t):=|t|^{r-2}t,
\qquad t\in\mathbb R,\quad r>1.
\]
Recall that $\mathfrak a_q:\mathbb R\to\mathbb R$
is a bijection whose inverse is $\mathfrak a_{q'}$, that is,
\[
\mathfrak a_{q'}(\mathfrak a_q(t))=t
\quad\text{and}\quad
\mathfrak a_q(\mathfrak a_{q'}(t))=t,
\qquad\forall t\in\mathbb R.
\]

\begin{proposition}[Second-order system formulation]
\label{prop:system-equivalence}
Let \(1<p<N\) and \(1<q<N/2\), and let
\(f(x,t)\) be such that all the distributional identities below are
well defined. Then the following statements are equivalent.

\begin{enumerate}[label=\textnormal{(\roman*)}]
\item There exist $u\in D^{1,p}(\mathbb R^N)\cap L^{q^{**}}(\mathbb R^N)$ and  $w\in L^{q'}(\mathbb R^N),$
such that
\begin{equation}
\label{eq:second-order-system}
\begin{cases}
\Delta w-\Delta_pu=f(x,u)
&\text{in }\mathcal D'(\mathbb R^N),\\[1mm]
\Delta u=|w|^{q'-2}w
&\text{in }\mathcal D'(\mathbb R^N).
\end{cases}
\end{equation}

\item There exists $u\in D^{1,p}(\mathbb R^N)\cap D^{2,q}(\mathbb R^N)$ such that
\begin{equation}
\label{eq:mixed-fourth-order}
\Delta\bigl(|\Delta u|^{q-2}\Delta u\bigr) -\Delta_pu =f(x,u)
\qquad\text{in }\mathcal D'(\mathbb R^N).
\end{equation}
\end{enumerate}
Moreover, the correspondence between the two formulations is given by $w=|\Delta u|^{q-2}\Delta u.$
\end{proposition}

\begin{proof}
Assume first that \textnormal{(i)} holds. Since \(w\in L^{q'}(\mathbb R^N)\),
we have $|w|^{q'-2}w\in L^q(\mathbb R^N).$ Indeed, $\bigl||w|^{q'-2}w\bigr|^q
= |w|^{q'},$ because $(q'-1)q=q'.$ 
The second equation in \eqref{eq:second-order-system} therefore gives  $\Delta u\in L^q(\mathbb R^N).$ Since \(u\in L^{q^{**}}(\mathbb R^N)\), the standard characterization of the homogeneous second-order Sobolev space yields $u\in D^{2,q}(\mathbb R^N).$
Applying \(\mathfrak a_q\) to the identity $\Delta u=\mathfrak a_{q'}(w)$ and using \(\mathfrak a_q\circ \mathfrak a_{q'}=\operatorname{id}_{\mathbb R}\), we obtain $w=\mathfrak a_q(\Delta u) = |\Delta u|^{q-2}\Delta u.$ a.e. in $\R^N$. Substituting this identity into the first equation of
\eqref{eq:second-order-system} gives \eqref{eq:mixed-fourth-order}. Thus \textnormal{(ii)} holds.

Conversely, assume that \textnormal{(ii)} holds and define $w:=|\Delta u|^{q-2}\Delta u.$ Since \(\Delta u\in L^q(\mathbb R^N)\), it follows that
\(w\in L^{q'}(\mathbb R^N)\). More precisely, $|w|^{q'} = |\Delta u|^{(q-1)q'} =|\Delta u|^q,$ because $(q-1)q'=q.$ Moreover, since \(\mathfrak a_{q'}=\mathfrak a_q^{-1}\), we have
\[
|w|^{q'-2}w=\mathfrak a_{q'}(w)
=\mathfrak a_{q'}(\mathfrak a_q(\Delta u))
=\Delta u
\]
pointwise almost everywhere. Hence
\[
\Delta u=|w|^{q'-2}w
\qquad\text{in }\mathcal D'(\mathbb R^N).
\]
Finally, equation \eqref{eq:mixed-fourth-order} becomes
\[
\Delta w-\Delta_pu=f(x,u)
\qquad\text{in }\mathcal D'(\mathbb R^N).
\]
Therefore \((u,w)\) satisfies \eqref{eq:second-order-system}, and
\textnormal{(i)} follows.
\end{proof}

We record a form of the
regularity-lifting principle used below.  Contraction-based
regularity-lifting lemmas for nonlinear integral equations and
systems can be found, for example, in \cite{LiXu2017, MaChenLi2011}.  The version needed here is adapted to a pointwise inequality and only requires positivity, order preservation, positive homogeneity, and subadditivity.

\begin{lemma}
\label{lem:unified-abstract-lifting}
Let \(\Omega\subset\mathbb R^N\) be measurable and let \(1<r<\tau<\infty\). Suppose that \(\mathcal R\) is positive, order preserving, positively
homogeneous and subadditive:
\[
  \mathcal R(v_1+v_2)
  \leq
  \mathcal R(v_1)+\mathcal R(v_2)
\]
for all nonnegative \(v_1,v_2\). Assume also that
\[
  \|\mathcal R(v)\|_{L^\ell(\Omega)}
  \leq
  \kappa_\ell\|v\|_{L^\ell(\Omega)},
  \qquad \ell=r,\tau,
\]
where \(max\{\kappa_r,\kappa_\tau\}<1.\) If \(v\in L^r(\Omega)\), \(h\in L^\tau(\Omega)\), \(v,h\ge 0\) and
\begin{equation}\label{pointwize ineq}
    v\leq\mathcal R(v)+h
  \qquad\text{a.e. in }\Omega,
\end{equation}
then \(v\in L^\tau(\Omega)\). Moreover,
\[
  \|v\|_{L^\tau(\Omega)}
  \leq
  \frac{1}{1-\kappa_\tau}\|h\|_{L^\tau(\Omega)}.
\]
\end{lemma}

\begin{proof}
By positivity, order preservation and subadditivity, iteration of \eqref{pointwize ineq} gives
\[v\leq
  \mathcal R^m(v)+ \sum_{j=0}^{m-1}\mathcal R^j(h) \qquad\text{a.e. in }\Omega.
\]
Since \(\mathcal R\) is a contraction on \(L^r(\Omega)\), $  \|\mathcal R^m(v)\|_{L^r(\Omega)}
  \leq
  \kappa_r^m\|v\|_{L^r(\Omega)}
  \rightarrow0.$ After passing to a subsequence,
\(\mathcal R^m(v)\to0\) a.e. in $\Omega$. 
The series on the right converges in \(L^\tau(\Omega)\), because
\[
  \sum_{j=0}^{\infty}
  \|\mathcal R^j(h)\|_{L^\tau(\Omega)}
  \leq
  \sum_{j=0}^{\infty}\kappa_\tau^j\|h\|_{L^\tau(\Omega)}
  =
  \frac{1}{1-\kappa_\tau}\|h\|_{L^\tau(\Omega)}.
\]
Therefore,
\(
0\le v
  \leq
  \sum_{j=0}^{\infty}\mathcal R^j(h)\)
 a.e. in \(\Omega,\)
and the conclusion follows.
\end{proof}

We use the Riesz potentials
\[
  I_\beta(g)(x) := c_{N,\beta}\int_{\mathbb R^N} \frac{g(y)}{|x-y|^{N-\beta}}\,dy,
  \qquad 0<\beta<N.
\]
We shall use the following standard global consequence of the
Hardy--Littlewood--Sobolev inequality.

\begin{lemma}
\label{lem:compact-source-riesz}
Let \(\beta\in(0,N)\), \(r\in(1,\infty)\) and suppose that
\(g\in L^r(\mathbb R^N)\) has compact support. If \(r<N/\beta\), then
\[
  I_\beta(g)\in L^\tau(\mathbb R^N)\quad\text{for every}\quad  \frac{N}{N-\beta}<\tau\leq
  \frac{Nr}{N-\beta r}.
\]
If \(r\geq N/\beta\), the same conclusion holds for every finite $\tau>{N}/({N-\beta}).$
\end{lemma}

\begin{proof}
On a fixed ball containing the support of \(g\), the assertion follows from the Hardy--Littlewood--Sobolev inequality. Outside a large ball,
\[
  I_\beta(g)(x)
  \leq
  C|x|^{\beta-N}\|g\|_{L^1},
\]
and this function belongs to \(L^\tau\) precisely when
\(\tau>N/(N-\beta)\). Combining the local and exterior estimates proves
the result.
\end{proof}

All the functions to which these operators are applied below are
nonnegative.
For measurable functions \(a\geq0\) and $v$, define
\begin{equation*}
  \mathcal T_a(v):= I_2\left[\left(I_2(a|v|^{q-1})\right)^{1/(q-1)}\right].
\end{equation*}
The composition defining $  \mathcal T_a$ has the structure of a
Havin--Maz'ya potential.  Indeed, for
$d\mu_v=a|v|^{q-1}\,dx$, it takes the form
\[
  \mathcal T_a(v)
= I_2\left[
  \bigl( I_2\mu_v\bigr)^{1/(q-1)}
\right].
\]
Such nonlinear compositions of Riesz potentials arise naturally
in nonlinear potential theory; see
\cite{CianchiSchwarzacher2018,MR0409858,KuusiMingione2014}.
Here, however, $\mu_v$ depends on the unknown itself, and
the estimate needed below follows directly from two applications of
the Hardy--Littlewood--Sobolev inequality.
We have the following lemma.

\begin{lemma}
Let $1<q<N/2$ and
\begin{equation}\label{eq:value.of.t}
\tau>\max\left\{\frac{N(q-1)}{N-2q},\frac{N}{N-2} \right\}.
\end{equation}
If \(a\in L^{N/(2q)}(\mathbb R^N)\), $a\ge 0$ and
\(v\in L^\tau(\mathbb R^N)\), then
\(\mathcal T_a(v)\in L^\tau(\mathbb R^N)\) and
\begin{equation}\label{eq:Ta-bound}
\|\mathcal T_a(v)\|_{\tau}
\le C_\tau \|a\|_{\frac{N}{2q}}^{1/(q-1)}\|v\|_{\tau}.
\end{equation}
\end{lemma}

\begin{proof}
Consider \(m>0\) and \(r>0\) such that
\[
  \frac1m:=\frac{2q}{N}+\frac{q-1}{\tau},
  \qquad
  \frac1r:=\frac2N+\frac1\tau.
\]
The assumptions on \(\tau\) ensure that \( 1<m,r< N/2\).
Hölder's inequality gives
\[\|a|v|^{q-1}\|_{m}\leq \|a\|_{\frac{N}{2q}}\|v\|_{\tau}^{q-1}.\]
Since \((q-1)/{r}=1/m-2/N\), by the Hardy--Littlewood--Sobolev inequality we obtain
\[
  \left\|
  I_2(a|v|^{q-1})
  \right\|_{\frac{r}{q-1}}
  \leq
  C\|a|v|^{q-1}\|_{m}\le C
  \|a\|_{\frac{N}{2q}}
  \|v\|_{\tau}^{q-1}.
\]
Consequently,
\[
  \left\|
  \left(I_2(a|v|^{q-1})\right)^{1/(q-1)}
  \right\|_{r}
  \leq
  C
  \|a\|_{\frac{N}{2q}}^{1/(q-1)}
  \|v\|_{\tau}.
\]
Since $1/\tau=1/r-2/N$, a second application of the Hardy--Littlewood--Sobolev inequality yields \eqref{eq:Ta-bound}.
\end{proof}

The operator \(\mathcal T_a\) is positively homogeneous for every
\(q>1\). It is genuinely subadditive when \(q\geq2\) (see Lemma below), but this generally fails when \(1<q<2\). We now modify the dependent variable so that the resulting operator is subadditive in both cases.

Denote
\begin{equation}
\label{eq:alpha-q}
  \alpha_q:=\min\{1,q-1\}.
\end{equation}
For  measurable functions \(a,z\geq0\), define
\begin{equation*}
  \mathcal R_a(z)
  :=
  \left[
  \mathcal T_a\left(z^{1/\alpha_q}\right)
  \right]^{\alpha_q}.
\end{equation*}
Let us also fix the notation \(L^{r}_+(\mathbb R^N) :=\left\{z\in L^{r}(\mathbb R^N):z\ge0\ \text{a.e.}\right\}.\)

\begin{lemma}\label{lem:unified-operator}
Let $1<q<N/2$ and $a\in L^{N/(2q)}_+(\mathbb R^N)$. If \(\tau\) satisfies \eqref{eq:value.of.t}, the operator $\mathcal R_a$ satisfies
\[
\mathcal R_a:L^{\tau/\alpha_q}_+(\mathbb R^N)
\rightarrow L^{\tau/\alpha_q}_+(\mathbb R^N),
\]
is  order preserving, positively homogeneous, and subadditive. Moreover, 
\begin{equation}\label{eq:Ra-bound}
\|\mathcal R_a(z)\|_{\frac{\tau}{\alpha_q}}
\le
C_\tau^{\alpha_q}
\|a\|_{\frac{N}{2q}}^{\alpha_q/(q-1)}
\|z\|_{\frac{\tau}{\alpha_q}}
\end{equation}
for every \(z\in L^{\tau/\alpha_q}_+(\mathbb R^N)\).
\end{lemma}

\begin{proof}
Positivity and order preservation are immediate. Since
\(\mathcal T_a\) is positively homogeneous,
\[
  \mathcal R_a(t z)
  =
 t\mathcal R_a(z)
  \qquad\text{for every }t\geq0.
\]
We prove subadditivity separately in the two ranges.
Suppose first that \(q\geq2\). Then \(\alpha_q=1\), and hence
\(\mathcal R_a=\mathcal T_a\). For every fixed \(x\), introduce the measure
\[
  d\mu_x(\xi):=\frac{c_{N,2}a(\xi)}{|x-\xi|^{N-2}}\,d\xi.
\]
Set \(\rho:=q-1\geq1\). Minkowski's inequality gives
\[
  \left(I_2\bigl(a(v_1+v_2)^\rho\bigr)(x)
  \right)^{1/\rho} \leq\left(I_2(av_1^\rho)(x)\right)^{1/\rho}
  + \left(I_2(av_2^\rho)(x)\right)^{1/\rho}.
\]
Applying the positive linear operator \(I_2\) proves
\[
  \mathcal R_a(v_1+v_2)
  \leq
  \mathcal R_a(v_1)+\mathcal R_a(v_2).
\]
Suppose now that \(1<q<2\). Then \(\alpha_q=\rho\in(0,1)\), and
\[
  \mathcal R_a(z)
  =
  \left\{
  I_2\left[
  \left(I_2(az)\right)^{1/\rho}
  \right]
  \right\}^{\rho}.
\]
Set
\(m:=1/\rho>1.\) For every fixed \(x\), with
\[
  d\nu_x(\xi)
  :=
  \frac{c_{N,2}}{|x-\xi|^{N-2}}\,d\xi,
\]
we may write
\[
  \mathcal R_a(z)(x)
  =
  \left(
  \int_{\mathbb R^N}
  \bigl[I_2(az)(\xi)\bigr]^m\,d\nu_x(\xi)
  \right)^{1/m}.
\]
Because \(I_2\) is linear and positive, Minkowski's inequality in
\(L^m(d\nu_x)\) yields
\[
  \mathcal R_a(z_1+z_2)(x)
  \leq
  \mathcal R_a(z_1)(x)
  +
  \mathcal R_a(z_2)(x).
\]
Thus \(\mathcal R_a\) is subadditive also when \(1<q<2\).
Finally, put \(v=z^{1/\alpha_q}\). By \eqref{eq:Ta-bound},
\[
  \begin{aligned}
  \|\mathcal R_a(z)\|_{\frac{\tau}{\alpha_q}}
  =
  \|\mathcal T_a(v)\|_{\tau}^{\alpha_q}
  \leq
  C_\tau^{\alpha_q}
  \|a\|_{\frac{N}{2q}}^{\alpha_q/(q-1)}
  \|v\|_{\tau}^{\alpha_q}
  =C_\tau^{\alpha_q}
  \|a\|_{\frac{N}{2q}}^{\alpha_q/(q-1)}
  \|z\|_{\frac{\tau}{\alpha_q}},
  \end{aligned}
\]
which proves \eqref{eq:Ra-bound}.
\end{proof}

 Let $u\in \mathcal W$ be a weak solution of
\eqref{generalfproblem}, and set $w:=|\Delta u|^{q-2}\Delta u.$ Let $\Omega\subset\mathbb R^N$ be any open set and fix bounded domains
\[
\Omega_0\Subset\Omega_1\Subset\Omega,
\]
and choose a nonnegative cutoff function
\[
\eta\in C_c^\infty(\Omega_1),
\qquad
\eta=1
\quad\text{in a neighborhood of }\overline{\Omega_0},
\]
such that every positive power of $\eta$ is smooth. Suppose $f$ satisfies \eqref{condf_0} and let
$C_{\Omega_1}$ be the constant in the local critical-growth condition
corresponding to $\Omega_1$. Define
\begin{equation}\label{V function}
V:=C_{\Omega_1}|u|^{q^{**}-q}\chi_{\Omega_1}.
\end{equation}
Since $u\in L^{q^{**}}(\Omega_1)$ and
\[
\frac{q^{**}}{q^{**}-q}=\frac{N}{2q},
\]
we have $V\in L^{N/(2q)}(\mathbb R^N)$. For $k>0$, set
\[
V_k:=V\chi_{\{V>k\}},
\qquad
V^k:=V\chi_{\{V\leq k\}}.
\]
Then $V^k\in L^\infty(\mathbb R^N)$, while
\begin{equation}
\label{eq:VK-small}
\|V_k\|_{\frac{N}{2q}}\rightarrow0
\qquad\text{as }k\rightarrow\infty.
\end{equation}

Set $\theta:=\eta^{q-1}$, and define
\begin{equation}
\label{eq:cutoff-unknowns}
U:=\eta u,
\qquad
W:=\theta w.
\end{equation}
We regard $U$ and $W$ as functions on $\mathbb R^N$ by extending
them by zero outside $\Omega_1$. The apparently asymmetric powers in
\eqref{eq:cutoff-unknowns} are essential. They give the identities
\begin{equation}
\label{eq:cutoff-power-identities}
\eta|w|^{q'-2}w=|W|^{q'-2}W,
\qquad
C_{\Omega_1}\theta|u|^{q^{**}-1}=V|U|^{q-1}.
\end{equation}

\begin{lemma}
\label{lem:closed-cutoff-inequality}
For every $k>0$, there exist an exponent $r_0>q^{**}$, depending only on $N,p,q$, and a nonnegative function
$G_k\in L^{r_0}(\mathbb R^N)$ such that
\begin{equation}
\label{eq:closed-cutoff-inequality}
|U|\leq C\mathcal T_{V_k}(|U|)+G_k
\qquad\text{a.e. in }\mathbb R^N.
\end{equation}
\end{lemma}

\begin{proof}
By Proposition~\ref{prop:system-equivalence}, the pair $(u,w)$
satisfies \eqref{eq:second-order-system}. The second equation in
\eqref{eq:second-order-system}, the product rule and
\eqref{eq:cutoff-power-identities} give
\[
\Delta U
=|W|^{q'-2}W+2\operatorname{div}(u\nabla\eta)-u\Delta\eta
\qquad\text{in }\mathcal D'(\mathbb R^N).
\]
We apply $I_2$ in this identity. Since $U$ is a compactly supported
distribution, $I_2(\Delta U)$ is well defined and
\[
I_2(\Delta U)=(\Delta I_2)*U=-U.
\]
On the other hand, each term on the right-hand side is represented by
a locally integrable Newtonian potential, with the divergence
producing, after integration by parts, an $I_1$-term. Consequently,
\begin{equation}
\label{eq:cutoff-U-potential}
|U|\leq CI_2\bigl(|W|^{q'-1}\bigr)+d(x),
\end{equation}
where
\begin{equation}
\label{eq:cutoff-B}
d(x):=CI_1(|u\nabla\eta|)+CI_2(|u\Delta\eta|).
\end{equation}
Let $h:=|\nabla u|^{p-2}\nabla u$. Since
\(
\theta\operatorname{div}h
=\operatorname{div}(\theta h)-h\cdot\nabla\theta,
\)
we obtain
\[
\Delta W
=\operatorname{div}(\theta h+2w\nabla\theta)
+\theta f(x,u)-h\cdot\nabla\theta-w\Delta\theta
\qquad\text{in }\mathcal D'(\mathbb R^N).
\]
The local critical-growth condition \eqref{condf_0} and
\eqref{eq:cutoff-power-identities} give
\[
\theta|f(x,u)|
\leq C_{\Omega_1}\theta+V|U|^{q-1}.
\]
Since $W$ is also a compactly supported distribution, splitting
$V=V_k+V^k$ in its equation and applying the Newtonian potential
estimate yield
\begin{equation}
\label{eq:cutoff-W-potential}
|W|\leq CI_2\bigl(V_k|U|^{q-1}\bigr)+R_k,
\end{equation}
where
\begin{align}
R_k(x):={}&
CI_1\bigl(\theta|h|+2|w\nabla\theta|\bigr)
+CI_2\bigl(|h||\nabla\theta|+|w\Delta\theta|\bigr)
\notag\\
&+CI_2\left(\theta+V^k|U|^{q-1}\right).
\label{eq:cutoff-RK}
\end{align}
We claim that there exists $\tau_0>q'$, independent of $k$, such that
\(
q<\tau_0(q-1)<N/2
\)
and
\begin{equation}
\label{eq:cutoff-RK-better}
R_k\in L^\tau(\mathbb R^N)
\qquad\forall\tau\in(q',\tau_0].
\end{equation}
Let us prove the claim. All inputs in \eqref{eq:cutoff-RK} have
compact support in $\Omega_1$. Since
$h\in L^{p'}(\Omega_1)$, Lemma~\ref{lem:compact-source-riesz} gives,
if $p'<N$,
\[
I_1(\theta|h|)\in L^\tau(\mathbb R^N)
\quad\text{for every}\quad
\frac{N}{N-1}<\tau\leq\tau_p,
\qquad
\frac1{\tau_p}=\frac1{p'}-\frac1N.
\]
Moreover,
\[
\frac1{\tau_p}-\frac1{q'}
=\frac1q-\frac1p-\frac1N<0,
\]
where the last inequality is equivalent to $p<q^*$. Hence
$\tau_p>q'$. If $p'\geq N$, Lemma~\ref{lem:compact-source-riesz}
instead gives
\[
I_1(\theta|h|)\in L^\tau(\mathbb R^N)
\qquad\forall \tau>\frac{N}{N-1}.
\]
Thus $I_1(\theta|h|)$ belongs to $L^\tau(\mathbb R^N)$ for every
$\tau>q'$ sufficiently close to $q'$.
The term $I_2(|h||\nabla\theta|)$ has better integrability. Indeed,
$h|\nabla\theta|\in L^{p'}(\mathbb R^N)$ and has compact support. If
$p'<N/2$, Lemma~\ref{lem:compact-source-riesz} gives
\[
I_2(|h||\nabla\theta|)\in L^\tau(\mathbb R^N)
\quad\text{for every}\quad
\frac{N}{N-2}<\tau\leq\tau_{p,2},
\qquad
\frac1{\tau_{p,2}}=\frac1{p'}-\frac2N,
\]
and $\tau_{p,2}>\tau_p>q'$. If $p'\geq N/2$, the same lemma gives
the conclusion for every finite $\tau>N/(N-2)$. Since $q<N/2$
implies $q'>N/(N-2)$, in either case this potential belongs to
$L^\tau(\mathbb R^N)$ for every $\tau>q'$ sufficiently close to
$q'$.

The terms involving $w$ admit the same improvement. Since
$w\in L^{q'}(\Omega_1)$ and the derivatives of $\theta$ are bounded
and compactly supported, Lemma~\ref{lem:compact-source-riesz} gives
\[
I_1(|w\nabla\theta|)\in L^\tau(\mathbb R^N)
\quad\text{for every}\quad
\frac{N}{N-1}<\tau\leq\tau_q,
\qquad
\frac1{\tau_q}=\frac1{q'}-\frac1N,
\]
if $q'<N$, and for every finite $\tau>N/(N-1)$ if $q'\geq N$.
Likewise,
\[
I_2(|w\Delta\theta|)\in L^\tau(\mathbb R^N)
\quad\text{for every}\quad
\frac{N}{N-2}<\tau\leq\tau_{q,2},
\qquad
\frac1{\tau_{q,2}}=\frac1{q'}-\frac2N,
\]
if $q'<N/2$, and for every finite $\tau>N/(N-2)$ if $q'\geq N/2$.
Whenever they are finite, $\tau_q,\tau_{q,2}>q'$. Hence both
potentials belong to $L^\tau(\mathbb R^N)$ for every $\tau>q'$
sufficiently close to $q'$.

Finally, $I_2(\theta)\in L^\tau(\mathbb R^N)$ for every
$\tau>N/(N-2)$, while
\[
V^k|U|^{q-1}\leq k\theta|u|^{q-1}.
\]
Since $\theta|u|^{q-1}
\in L^{q^{**}/(q-1)}(\mathbb R^N)$ and has compact support, Lemma~\ref{lem:compact-source-riesz} shows
that $I_2(V^k|U|^{q-1})$ belongs to $L^\tau(\mathbb R^N)$ for every
$\tau>q'$ sufficiently close to $q'$. Indeed, if
$(q-1)/q^{**}-2/N>0$, the upper exponent $\widetilde\tau_q$ given by
\[
\frac1{\widetilde\tau_q}
=\frac{q-1}{q^{**}}-\frac2N
\]
satisfies
\[
\frac1{\widetilde\tau_q}-\frac1{q'}=-\frac{2q}{N}<0,
\]
and otherwise the conclusion holds for every finite
$\tau>N/(N-2)$. This proves \eqref{eq:cutoff-RK-better} after choosing
$\tau_0>q'$ sufficiently close to $q'$. Such a choice can also be
made so that 
$$q<\tau_0(q-1)<N/2,$$
because $q<N/2$. The corresponding
norm of $R_k$ may depend on $k$.

Since $q'-1=1/(q-1)$, \eqref{eq:cutoff-W-potential} gives
\[
|W|^{q'-1}
\leq
C\left[I_2(V_k|U|^{q-1})\right]^{1/(q-1)}
+CR_k^{1/(q-1)}.
\]
Substituting this estimate into \eqref{eq:cutoff-U-potential}, we
obtain
\[
|U|\leq C\mathcal T_{V_k}(|U|)
+CI_2\left(R_k^{1/(q-1)}\right)+d(x).
\]
By \eqref{eq:cutoff-RK-better},
$R_k^{1/(q-1)}\in L^{\tau(q-1)}(\mathbb R^N)$ for every
$q'<\tau\leq\tau_0$, where $\tau_0(q-1)<N/2$. Hence the
Hardy--Littlewood--Sobolev inequality gives
\[
I_2\left(R_k^{1/(q-1)}\right)
\in L^r(\mathbb R^N)
\quad\forall r\in(q^{**},r_1],
\qquad
\frac1{r_1}:=\frac1{\tau_0(q-1)}-\frac2N.
\]

It remains to estimate $d$. Since $u\nabla\eta$ and $u\Delta\eta$
belong to $L^{q^{**}}(\mathbb R^N)$ and have compact support,
Lemma~\ref{lem:compact-source-riesz} shows that
\[
I_1(|u\nabla\eta|)
\quad\text{and}\quad
I_2(|u\Delta\eta|)
\]
belong to $L^r(\mathbb R^N)$ for every $r>q^{**}$ sufficiently close
to $q^{**}$. Consequently, we may choose $q^{**}<r_0\leq r_1,$ independently of $k$, such that
\[
I_2\left(R_k^{1/(q-1)}\right),d\in L^{r_0}(\mathbb R^N).
\]
Defining
\[
G_k:=CI_2\left(R_k^{1/(q-1)}\right)+d(x),
\]
we obtain $G_k\in L^{r_0}(\mathbb R^N)$ and
\eqref{eq:closed-cutoff-inequality} follows.
\end{proof}

We can now implement the first regularity lifting.  The underlying strategy is to separate the critical coefficient into a part with
small norm, which generates a contraction, and a bounded part that is absorbed into the more regular remainder.  Related contraction arguments for critical integral systems appear in \cite{MaChenLi2011}.  In the present problem, the localized second-order system closes instead through the nonlinear composition $\mathcal T_{V_k}$ constructed above.

\begin{proposition}\label{thm:unified-critical-gain}
Let $u\in \mathcal W$ be a weak solution of
\eqref{generalfproblem} and for a fixed open set $\Omega\subset\mathbb R^N$, assume that \eqref{condf_0} holds. Then there exists $\varepsilon>0$ such
that $u\in L_{\mathrm{loc}}^{q^{**}+\varepsilon}(\Omega).$

\end{proposition}

\begin{proof}
Fix bounded domains
$\Omega_0\Subset\Omega_1\Subset\Omega$, and let $U,V_k,G_k$ be the
functions constructed above. Let $\alpha_q$ be given by
\eqref{eq:alpha-q} and define $Z:=|U|^{\alpha_q}$. Since
$0<\alpha_q\leq1$, \eqref{eq:closed-cutoff-inequality} yields
\[
Z
\leq
C^{\alpha_q}
\bigl[\mathcal T_{V_k}(|U|)\bigr]^{\alpha_q}
+G_k^{\alpha_q}
=C^{\alpha_q}\mathcal R_{V_k}(Z)+G_k^{\alpha_q}.
\]
Define
$\widetilde{\mathcal R}_k:=C^{\alpha_q}\mathcal R_{V_k}$, and set $r:={q^{**}}/{\alpha_q}$ and $\tau:={r_0}/{\alpha_q},$ where $r_0$ is given by Lemma~\ref{lem:closed-cutoff-inequality}.
Then $1<r<\tau<\infty$, and
\[
Z\in L^r(\mathbb R^N),
\qquad
G_k^{\alpha_q}\in L^\tau(\mathbb R^N).
\]
By Lemma~\ref{lem:unified-operator}, the operator
$\widetilde{\mathcal R}_k$ is positive, order preserving, positively
homogeneous and subadditive. Moreover, for
$\ell=q^{**},r_0$, the same lemma gives
\[
\|\widetilde{\mathcal R}_k(z)\|_{L^{\ell/\alpha_q}}
\leq
C^{\alpha_q}C_\ell^{\alpha_q}
\|V_k\|_{L^{N/(2q)}}^{\alpha_q/(q-1)}
\|z\|_{L^{\ell/\alpha_q}}.
\]
In view of \eqref{eq:VK-small}, $k$ can be chosen so large that
\[
C^{\alpha_q}
\max\{C_{q^{**}}^{\alpha_q},C_{r_0}^{\alpha_q}\}
\|V_k\|_{L^{N/(2q)}}^{\alpha_q/(q-1)}<1.
\]
Lemma~\ref{lem:unified-abstract-lifting} therefore gives
$Z\in L^{r_0/\alpha_q}(\mathbb R^N)$. This is exactly
$U\in L^{r_0}(\mathbb R^N)$. Since $U=u$ on $\Omega_0$, we conclude
that $u\in L^{r_0}(\Omega_0).$ The exponent $r_0>q^{**}$ is independent of the domains
$\Omega_0,\Omega_1$. Since $\Omega_0\Subset\Omega$ is arbitrary,
taking $\varepsilon:=r_0-q^{**}>0$ proves the assertion.
\end{proof}

The following auxiliary result will be used  repeatedly in the proof of the local regularity theorem.

\begin{lemma}\label{lem:local-W1-div}
Let $\Omega\subset\mathbb R^N$ be open, let $F\in L^a_{\rm loc}(\Omega,\mathbb R^N)$ and  $g\in L^b_{\rm loc}(\Omega)$, with $1<a,b<\infty$, and let $v\in L^1_{\rm loc}(\Omega)$ satisfy
\[
\Delta v=\operatorname{div}F+g
\]
in the sense of distributions. Then $v\in W^{1,r}_{\rm loc}(\Omega)$ for every $1\le r\le\min\{a,b^*\}$, if $b< N$ or every  $1\le r\le a$ otherwise.

\end{lemma}

\begin{proof}
Fix domains $\Omega_0\Subset\Omega_1\Subset\Omega,$ where $\Omega_1$ has smooth boundary. Since $F\in L^a(\Omega_1,\mathbb R^N)$, then there exists a unique function $v_1\in W^{1,a}_0(\Omega_1)$
satisfying
\[
\int_{\Omega_1}\nabla v_1\cdot\nabla\varphi
=-\int_{\Omega_1}F\cdot\nabla\varphi,
\qquad
\forall\,
\varphi\in W^{1,a'}_0(\Omega_1),
\]
that is, $-\Delta v_1=\operatorname{div}F$ in $\Omega_1$. (See, for instance Theorem II.1.2 of Simader--Sohr \cite{SimaderSohr1996}). Likewise, since $g\in L^b(\Omega_1)$, the classical local regularity theory for the Poisson equation
(see, e.g., Gilbarg--Trudinger \cite{GilbargTrudinger1977}, Theorem 9.15)
provides a function
\[
v_2\in W^{2,b}(\Omega_1)\cap W^{1,b}_0(\Omega_1)
\]
such that $-\Delta v_2=g$ in $\Omega_1$.
Hence $\Delta(v+v_1+v_2)=0$ in the sense of distributions in $\Omega_1$.
Since $v+v_1+v_2\in L^1_{\rm loc}(\Omega_1)$, Weyl's lemma implies that $v+v_1+v_2\in C^\infty(\Omega_1).$

Notice that $v=(v+v_1+v_2)-v_1-v_2,$
with $v_1\in W^{1,a}(\Omega_1)$ and $v_2\in W^{2,b}(\Omega_1) \subset W^{1,b^*}(\Omega_1),$  in case $b<N$, while $v_2\in W^{1,r}(\Omega_1)$ for all finite $r$ in case $b\ge N$.
Consequently, $v\in W^{1,r}(\Omega_0)$ for every $1\le r\le\min\{a,b^*\}$ if $b<N$ or $1\le r\le a$ otherwise. Since $\Omega_0\Subset\Omega$ is arbitrary, the conclusion follows.
\end{proof}

Now we can work on a bootstrap argument to reach higher local regularity.

\begin{proof}[Proof of Theorem~\ref{thm:intro-regularity}]
Consider $u\in\mathcal{W}$ a weak solution for \eqref{generalfproblem}, with $f$ satisfying \eqref{condf_0}. 
We divide the proof into three steps.

\medskip

\noindent
\textbf{Step 1:} There exists \(c_1>q\) such that
\begin{equation}\label{u W2c1}
u\in W_{\mathrm{loc}}^{2,c_1}(\Omega).
\end{equation}

\medskip

\noindent To prove this step, let $\Omega_1\Subset\Omega_2\Subset \Omega.$ By Proposition \ref{thm:unified-critical-gain}, there exists $r_0>q^{**}$ such that $u\in L^{r_0}(\Omega_2).$
Due to Proposition \ref{prop:system-equivalence}, the pair \((u,w)\) satisfies
\begin{equation}
\label{eq:pohozaev-bootstrap-system}
\begin{cases}
\Delta w = \operatorname{div}\bigl(|\nabla u|^{p-2}\nabla u\bigr) + f(x,u),
\\[1mm]
\Delta u=|w|^{q'-2}w,
\end{cases}
\end{equation}
in the sense of distributions.
Since $u\in D^{2,q}(\mathbb R^N),$
we have $\nabla u\in L_{\mathrm{loc}}^{q^*}(\mathbb R^N)$ and so
\[
|\nabla u|^{p-2}\nabla u
\in L_{\mathrm{loc}}^{a_0}(\mathbb R^N),
\qquad
a_0:=\frac{q^*}{p-1}.
\]
Notice that $p< q^*$ implies that $a_0>[(q^{**})']^*$.
We also have
\[
|u|^{q^{**}-2}u
\in L^{b_0}(\Omega_1),
\qquad
b_0:=\frac{r_0}{q^{**}-1}.
\]
Therefore, due to \eqref{condf_0}, $f(x,u)\in L^{b_0}(\Omega_1)$. Since \(r_0>q^{**}\), we get $b_0>(q^{**})'$ and so $b_0^*>[(q^{**})']^*$.
 Since $p<q^*$, we verify that $a_0>[(q^{**})']^*$ as well.
By Lemma \ref{lem:local-W1-div} it follows that $w\in W^{1,s_0}(\Omega_1)$ for $s_0=\min\{a_0,b_0^*\}>[(q^{**})']^*.$ 
If \(s_0<N\), the Sobolev embedding gives
$w\in L^{s_0^*}(\Omega_1)$.
Note that \(
s_0^*>[(q^{**})']^{**}=q'.
\)
 If \(s_0\ge N\), the corresponding Sobolev embedding instead yields
\(w\in L^d(\Omega_1)\) for every finite \(d\). In either case, there exists
\(
d_0>q'
\)
such that
\(
w\in L^{d_0}(\Omega_1).
\)
Therefore,
\[
|w|^{q'-2}w
\in L^{c_1}(\Omega_1),
\qquad
c_1:=(q-1)d_0>(q-1)q'=q.
\]
Since it works for any $\Omega_1\Subset\Omega_2\Subset \Omega$, we conclude that 
$
|w|^{q'-2}w
\in L^{c_1}_{\rm{loc}}(\Omega).
$
The second equation in \eqref{eq:pohozaev-bootstrap-system} and the
 Calderón--Zygmund estimate now give  \eqref{u W2c1}.

\medskip

Notice that,  if $c_1\ge N/2$, then $u\in L_{\mathrm{loc}}^{\tau}(\Omega)$, for all $\tau\in[1,\infty)$.  This motivates the next step.
\medskip

\noindent
\textbf{Step 2:} If $u\in W_{\mathrm{loc}}^{2,c}(\Omega)$ for some $c\ge c_1$ and $u\in L_{\mathrm{loc}}^{r}(\Omega)$, for all $r\in[1,\infty)$ then \eqref{all estimate} holds. 
\medskip

\noindent If $u\in L_{\mathrm{loc}}^{r}(\Omega)$, for all $r\in[1,\infty)$,  due to condition \eqref{condf_0} we have
\begin{equation}\label{f in Ls}
f(x,u)\in L^r_{\rm{loc}}(\Omega),\quad\forall r\in[1,\infty).
\end{equation}
If $u\in W_{\mathrm{loc}}^{2,c}(\Omega)$ we also have $|\nabla u|\in W_{\mathrm{loc}}^{1,c}(\Omega)$. If $c\ge N$, then we get $|\nabla u|\in L_{\mathrm{loc}}^{r}(\Omega),$ for all $r\ge 1$. On the other hand, if $q<c<N$, we see that $|\nabla u|\in L_{\mathrm{loc}}^{c^*}(\Omega)$. Let us define
\begin{equation}\label{a(c) def}
a(c)=c^*/(p-1),
\end{equation}
in the same spirit of $a_0$.  Then
\begin{equation}\label{nabla u L ac}
|\nabla u|^{p-2}\nabla u
\in\begin{cases}L^{a(c)}_{\mathrm{loc}}(\Omega),&\text{if }c<N,\\
L^{r}_{\mathrm{loc}}(\Omega),\quad\forall r\in[1,\infty),&\text{if }c\ge N.
\end{cases}
\end{equation}
Using \eqref{f in Ls} and \eqref{nabla u L ac}, we can apply Lemma~\ref{lem:local-W1-div} to the first equation in
\eqref{eq:pohozaev-bootstrap-system}, to obtain
\begin{equation}\label{cases wcN}
w\in\begin{cases}
W^{1,a(c)}_{\mathrm{loc}}(\Omega),&\text{if }c< N,
\\W^{1,m}_{\mathrm{loc}}(\Omega),\quad\forall m\in[1,\infty)&\text{if }c\ge N.\end{cases}
\end{equation}
Now, we distinguish three cases:

\bigskip

\noindent\textit{Case 1: $c\ge N$.}

\bigskip

\noindent From \eqref{cases wcN} we have $w\in W^{1,m}_{\mathrm{loc}}(\Omega)$ for any $m\in[1,\infty)$. So, 
\begin{equation}\label{wq'Lm}
    |w|^{q'-2}w\in L^m_{\mathrm{loc}}(\Omega)\quad\forall m\in[1,\infty).
\end{equation}
Thus, the second equation in
\eqref{eq:pohozaev-bootstrap-system} provides 
\begin{equation}\label{eq:u-regular-e}
u\in W^{2,m}_{\mathrm{loc}}(\Omega)\quad\forall m\in[1,\infty),
\end{equation}
which completes the proof in this case.

\bigskip

\noindent\textit{Case 2: $c<N$ and $a(c)\ge N$.}

\bigskip

\noindent In this case, from \eqref{cases wcN} we get $w\in L^r_{\mathrm{loc}}(\Omega)$ for all $r\in[1,\infty)$. Then it follows that
\eqref{wq'Lm} occurs. 
Again, the second equation in
\eqref{eq:pohozaev-bootstrap-system}
yields \eqref{eq:u-regular-e}.

Then, choosing \(m>N\) in \eqref{eq:u-regular-e} we see that $|\nabla u|\in L^\infty_{\mathrm{loc}}(\Omega)$. Thus, the local $r$-integrability of $f(x,u)$ for all finite $r\ge1$ and Lemma~\ref{lem:local-W1-div} give
\begin{equation*}
w\in W^{1,m}_{\mathrm{loc}}(\Omega)\quad\forall m\in[1,\infty).
\end{equation*}


\noindent\textit{Case 3: $c<N$ and $a(c)< N$.}

\medskip

\noindent Let us first notice that $a(c)<N$ if and only if $c<N/p^\prime$. This equivalence will be useful in the arguments below. Now, due to \eqref{cases wcN}, we obtain
\(
w\in W^{1,a(c)}_{\mathrm{loc}}(\Omega)
\hookrightarrow
L^{a(c)^*}_{\mathrm{loc}}(\Omega).
\)
Hence
\(
|w|^{q'-2}w
\in
L_{\mathrm{loc}}^{c_p(c)}(\Omega),
\)
for 
\begin{equation}\label{cp(c)}
c_p(c):=(q-1)a(c)^*.
\end{equation}
The second equation of
\eqref{eq:pohozaev-bootstrap-system}
and the interior \(W^{2,m}\)-estimate yield $u\in W^{2,c_p(c)}_{\mathrm{loc}}(\Omega).$
Notice that
\eqref{cp(c)} gives
\begin{align*}
\frac1{c_p(c)}-\frac1c
&=
\frac1c\frac{p-q}{q-1}
-
\frac{p}{N(q-1)}.
\end{align*}
If \(p\le q\), then
\begin{equation*}
\frac1{c_p(c)}-\frac1c
\le
-\frac{p}{N(q-1)}.
\end{equation*}
If \(q<p<q^*\), the fact that \(c\ge c_1>q\) gives
\begin{align*}
\frac1{c_p(c)}-\frac1c
&<\frac{p-q}{q(q-1)}
-\frac{p}{N(q-1)}
=-\frac1{q-1}\left(\frac pN-\frac{p-q}{q}\right).
\end{align*}
The quantity in parentheses is positive precisely when \(p<q^*\).
Accordingly, define
\begin{equation*}
\delta_p:=
\begin{cases}
\displaystyle\frac{p}{N(q-1)},
& p\le q,\\[2mm]
\displaystyle\frac1{q-1}
\left(
\frac pN-\frac{p-q}{q}
\right),
& q<p<q^*.
\end{cases}
\end{equation*}
Then \(\delta_p>0\) and, in either regime,
\begin{equation}
\label{eq:cp-uniform-improvement-both-regimes}
\frac1{c_p(c)}
\le
\frac1c-\delta_p.
\end{equation}

\noindent Denote $c_2=c_p(c_1)$. Then $u\in W_{\mathrm{loc}}^{2,c_2}(\Omega)$ and, from \eqref{eq:cp-uniform-improvement-both-regimes}, we get
\[
\frac{1}{c_2}-\frac{1}{c_1}\leq-\delta_p.
\]
If $c_2\ge N$ we go back to \textit{Case 1}. If $N>c_2\ge N/p^\prime$ then $a(c_2)\ge N$ and we go back to \textit{Case 2}, finishing the proof in either case.   If $c_2< N/p^\prime$, we repeat the process with $c_2$ instead of $c_1$. 
Following this idea,  
we may construct inductively a
sequence \((c_j)\), as long as \(c_j< N/p^\prime\), such that
\[
u\in W_{\mathrm{loc}}^{2,c_{j+1}}(\Omega)
\]
and
\[
\frac1{c_{j+1}} \le \frac1{c_j}-\delta_p,\qquad\text{which implies}\qquad \frac1{c_{j+1}}\le \frac1{c_1}-j\delta_p.
\]
Since the right-hand side becomes smaller than \(p^\prime/N\) after finitely many
steps, there must exist \(j_0\in\mathbb{N}\) such that $c_{j_0}\ge N/{p^\prime}.$ At this point,  the conclusion follows from \textit{Case 1}, if $c_{j_0}\ge N$, or \textit{Case 2}, if $N>c_{j_0}\ge N/p^\prime$.

\bigskip

\noindent
\textbf{Step 3:} Finite bootstrap up to an exponent larger than \(N/2\).

\bigskip
\noindent From \textbf{Step 1} we know that $u\in W^{2,c_1}_{\mathrm{loc}}(\Omega)$ for some $c_1>q$. If $c_1\geq N/2$ our proof would be complete, due to \textbf{Step 2}. Then, suppose that $c_1< N/2$ and assume that for some $q<c_1\leq c< N/2,$ we already know that
$
u\in W^{2,c}_{\mathrm{loc}}(\Omega).
$
Then
\[
|\nabla u|\in L_{\mathrm{loc}}^{c^*}(\Omega),
\qquad
u\in L_{\mathrm{loc}}^{c^{**}}(\Omega).
\]
Consequently, for $a(c)={c^*}/({p-1})$, as defined in \eqref{a(c) def}, we get $|\nabla u|^{p-2}\nabla u \in L_{\mathrm{loc}}^{a(c)}(\Omega).$
Likewise,
\[
|u|^{q^{**}-2}u
\in
L_{\mathrm{loc}}^{b(c)}(\Omega),
\qquad
b(c):=\frac{c^{**}}{q^{**}-1}.
\]
We have $f(x,u)\in L_{\mathrm{loc}}^{b(c)}(\Omega)$ due to \eqref{condf_0}.
Applying Lemma~\ref{lem:local-W1-div} to the first equation in
\eqref{eq:pohozaev-bootstrap-system}, we obtain
\begin{equation}\label{general c}
w\in
W^{1,s(c)}_{\mathrm{loc}}(\Omega),
\quad\text{for}\quad
s(c)
=
\min\{a(c),\,b(c)^*\}.
\end{equation}

Next we are going to study some cases, depending on  $s(c)$.
\smallskip

\noindent \emph{Case 4. Suppose that \(s(c)\ge N\).}

\medskip
\noindent If \(s(c)\ge N\), using \eqref{general c}, the corresponding Sobolev embedding yields
\(w\in L^m_{\rm{loc}}(\Omega)\) for every \(m\in[1,\infty)\), and so
\eqref{wq'Lm} holds. 
Again, the second equation in
\eqref{eq:pohozaev-bootstrap-system}
yields
\[
u\in W^{2,m}_{\mathrm{loc}}(\Omega),\quad\forall m\in[1,\infty).
\]
In this case, it follows from \textbf{Step 2}  that $w\in W^{1,m}_{\mathrm{loc}}(\Omega)$ for all $m\in[1,\infty)$, which means that \eqref{all estimate} occurs. This finishes the proof if $s(c)\ge N$.

\medskip

\noindent\emph{Case 5. Suppose that $s(c)<N$.} 

\medskip

\noindent If $a(c)\le b(c)^*$ we have
$
s(c)=a(c)<N,
$
and we argue as in \textit{Case 3} of \textbf{Step 2}. We see that \eqref{general c} implies
\[
w\in W^{1,a(c)}_{\mathrm{loc}}(\Omega)
\hookrightarrow
L^{a(c)^*}_{\mathrm{loc}}(\Omega).
\]
Hence, for $c_p(c)=(q-1)a(c)^*$ as in \eqref{cp(c)}, it holds
\[
|w|^{q'-2}w
\in
L_{\mathrm{loc}}^{c_p(c)}(\Omega).
\]
The second equation of
\eqref{eq:pohozaev-bootstrap-system}
and the interior \(W^{2,m}\)-estimate give
\[
u\in W^{2,c_p(c)}_{\mathrm{loc}}(\Omega).
\]
Moreover, as can be seen in \eqref{eq:cp-uniform-improvement-both-regimes} there is \(\delta_p>0\), independent of $c>q$, such that
\begin{equation*}
\frac1{c_p(c)}
\le
\frac1c-\delta_p.
\end{equation*}


\noindent On the other hand, if $b(c)^*< a(c)$ we have
$
s(c)=b(c)^*<N.
$
Hence, \eqref{general c} provides
\[
w
\in
W^{1,b(c)^*}_{\mathrm{loc}}(\Omega)
\hookrightarrow
L^{(b(c)^*)^*}_{\mathrm{loc}}(\Omega).
\]
Since
$(q-1)(q'-1)=1$, it follows that
\[
|w|^{q'-2}w
\in
L_{\mathrm{loc}}^{c_q(c)}(\Omega),
\]
where $c_q(c) := (q-1)(b(c)^*)^*.$ Again, the second equation in
\eqref{eq:pohozaev-bootstrap-system}
yields $u\in W^{2,c_q(c)}_{\mathrm{loc}}(\Omega).$ Using the expression for \(b(c)\), we have
\begin{equation}
\label{eq:cq-bootstrap-map}
\frac1{c_q(c)}
=
\frac{q^{**}-1}{q-1}
\left(
\frac1c-\frac2N
\right)
-
\frac2{N(q-1)}.
\end{equation}


\noindent Now, denoting $A:=(q^{**}-1)/(q-1)$ from 
\eqref{eq:cq-bootstrap-map} we obtain
\begin{equation*}
\frac1{c_q(c)}-\frac1q
=
A\left(
\frac1c-\frac1q
\right).
\end{equation*}
Since $c\ge c_1>q$, it follows  that
\begin{align}
\frac1{c_q(c)}-\frac1c
&=
(A-1)
\left(
\frac1c-\frac1q
\right)
\le
-(A-1)
\left(
\frac1q-\frac1{c_1}
\right).
\label{eq:cq-uniform-improvement}
\end{align}
Therefore, for $s(c)<N$,
we define
\begin{equation*}
\delta(c):=
\begin{cases}
\delta_p& a(c)\le b(c)^*,\\[2mm]
(A-1)
\left(
\frac1q-\frac1{c_1}
\right),
& b(c)^*<a(c).
\end{cases}
\end{equation*}
Notice that 
\begin{equation*}
\delta(c)\geq\delta_0:=\min\left\{\delta_p;(A-1)
\left(
\frac1q-\frac1{c_1}
\right)\right\}>0,  \quad\forall c_1\leq c<\frac N2.
\end{equation*}

\noindent If $a(c)\leq b(c)^*$, we denote $c_2=c_p(c_1)$. If $b(c)^*< a(c)$, we consider $c_2=c_q(c_1)$. In any case, $u\in W_{\mathrm{loc}}^{2,c_2}(\Omega)$ and, from \eqref{eq:cp-uniform-improvement-both-regimes} and \eqref{eq:cq-uniform-improvement}, we get
\[
\frac{1}{c_2}-\frac{1}{c_1}\leq-\delta(c_1)\le-\delta_0.
\]
If $c_2\ge N/2$ this proof is finished, by \textbf{Step 2}. If $c_2< N/2$ and $s(c_2)\ge N$, as in \textit{Case 4} this proof is complete. 
Following this idea,  
we may construct inductively a
sequence \((c_j)\), as long as \(c_j< N/2\) and \(s(c_j)< N\), such that $u\in W_{\mathrm{loc}}^{2,c_{j+1}}(\Omega)$ and $1/{c_{j+1}} \le 1/{c_j}-\delta_0.$ Consequently,
\[
\frac1{c_{j+1}}
\le
\frac1{c_1}-j\delta_0
\]
whenever $c_j<N/2$ and $s(c_j)<N$.
Since the right-hand side becomes smaller than \(2/N\) after finitely many
steps, there must exist \(j_0\in\mathbb{N}\) such that $c_{j_0}\ge N/2$ or $s(c_{j_0})\ge N.$ Then the conclusion follows from \textbf{Step 2} or \textit{Case 4}.
\end{proof}

\section{The Pohozaev identity}
\label{sec:pohozaev-identity}

We derive the Pohozaev identity from the annular regularity established above. The argument itself does not depend on the relative ordering of $p$ and $q$. The restrictions $q<N/2$ and $p<q^*$ enter only through the preceding regularity theory.

\begin{lemma}
\label{lem:mixed-hessian-flux}
Let \(\Omega\subset\mathbb R^N\) be open and suppose that
\[
u\in W_{\mathrm{loc}}^{2,m}(\Omega)\quad\text{and}\quad
w:=|\Delta u|^{q-2}\Delta u\in W_{\mathrm{loc}}^{1,m}(\Omega)
\quad\text{for some } m>N.
\]
Then, for every \(\xi\in C_c^\infty(\Omega)\),
\begin{align}
&-\int_\Omega\nabla w\cdot\nabla(\xi\,x\cdot\nabla u)\,dx
=\frac{2q-N}{q}\int_\Omega\xi|\Delta u|^q\,dx
-\frac1q\int_\Omega|\Delta u|^q x\cdot\nabla\xi\,dx
\nonumber\\
&\qquad
+2\int_\Omega w\nabla u\cdot\nabla\xi\,dx
+2\int_\Omega w\sum_{i,j=1}^Nu_{x_ix_j}x_j\xi_{x_i}\,dx
+\int_\Omega w(x\cdot\nabla u)\Delta\xi\,dx.
\label{eq:mixed-hessian-flux}
\end{align}
\end{lemma}

\begin{proof}
For smooth \(u,w\), expand the derivative on the left of \eqref{eq:mixed-hessian-flux} and integrate
each term containing a derivative of \(w\) by parts. This gives
\begin{align}
\nonumber&-\int_\Omega \nabla w\cdot\nabla(\xi\,x\cdot\nabla u)\,dx=
2\int_\Omega \xi w\Delta u\,dx
+\int_\Omega \xi w\,x\cdot\nabla(\Delta u)\,dx\\\label{eq:general-mixed-hessian}
&\qquad
+2\int_\Omega w\nabla u\cdot\nabla\xi\,dx
+2\int_\Omega w\sum_{i,j=1}^Nu_{x_ix_j}x_j\xi_{x_i}\,dx
+\int_\Omega w(x\cdot\nabla u)\Delta\xi\,dx.
\end{align}
Since \(w=|\Delta u|^{q-2}\Delta u\),
we have
$w\nabla(\Delta u)=\nabla(|\Delta u|^q)/q$. 
Integrating this last gradient by parts proves
\eqref{eq:mixed-hessian-flux} in the smooth case.

Now we prove the general result by approximation. Choose a smooth  open set $V$ such that $\operatorname{supp}\xi\Subset V\Subset \Omega$ and approximate \(w\) in \(W^{1,m}(V)\), with \(m>N\), by smooth
functions \(w_n\). Then \(w_n\to w\) uniformly on compact subsets of
\(V\), and the sequence is uniformly bounded there. Let
\(\beta_n\in C^\infty(\mathbb R)\) be  functions that
approximate $\beta(t)=|t|^{q'-2}t$ uniformly on bounded intervals. Define
\begin{align}\label{Bn(t)}
B_n(t) := t\beta_n(t)-\int_0^t\beta_n(\tau)\,d\tau. 
\end{align}
Then $B_n\to B$ uniformly on bounded intervals, where
\[
B(t):= t\beta(t)-\int_0^t\beta(\tau)\,d\tau =\frac1q|t|^{q'}.
\]
 The preceding uniform convergence and the
local uniform boundedness of \(w_n\) imply
\begin{equation}\label{zn to laplace u}
z_n:=\beta_n(w_n)\rightarrow\beta(w)=\Delta u\quad\text{in}\,\, L^\infty_{\mathrm{loc}}(V)\,\,\text{and in}\,\,L^m_{\mathrm{loc}}(V).
\end{equation} 
Choose a smooth domain \(V_0\) such that $\operatorname{supp}\xi\Subset V_0\Subset V.$ Let \(v_n\in W^{2,m}(V_0)\cap W_0^{1,m}(V_0)\) solve
\[
\Delta v_n=z_n-\Delta u
\quad\text{in }V_0,
\qquad
v_n=0\quad\text{on }\partial V_0,
\]
and set \(u_n:=u+v_n\). The Calderón--Zygmund estimate gives $v_n\rightarrow0$ in $W^{2,m}(V_0).$
Moreover,
\begin{align}
\Delta u_n=z_n
\quad\text{and}\quad
u_n\rightarrow u
\quad\text{in }\,W^{2,m}(V_0). \label{laplace un=zn}
\end{align}
Since $z_n$ is smooth, interior elliptic regularity implies that $u_n$ is smooth in $V_0$.
Thus, both $u_n$ and $w_n$ are smooth in $V_0$. Applying \eqref{eq:general-mixed-hessian} to $(u_n,w_n)$ and using \eqref{laplace un=zn}, gives
\begin{align}
&-\int_{V_0}
\nabla w_n\cdot\nabla(\xi\,x\cdot\nabla u_n)\,dx = 2\int_{V_0}\xi w_nz_n\,dx
+ \int_{V_0}\xi w_nx\cdot\nabla z_n\,dx
\nonumber\\
&\quad+
2\int_{V_0}w_n\nabla u_n\cdot\nabla\xi\,dx
+ 2\int_{V_0}w_n
\sum_{i,j=1}^N (u_n)_{x_ix_j}x_j\xi_{x_i}\,dx
+\int_{V_0}w_n(x\cdot\nabla u_n)\Delta\xi\,dx.
\label{eq:approximate-mixed-hessian}
\end{align}
Notice that \eqref{Bn(t)} gives
$B_n'(t)=t\beta_n'(t).$
Then, since $\nabla z_n
=\beta_n'(w_n)\nabla w_n$, we see that
\begin{equation}
\label{eq:approximate-chain-rule}
w_n\nabla z_n=w_n\beta_n'(w_n)\nabla w_n
=\nabla B_n(w_n).
\end{equation}
Recalling that $\xi$ has compact support in $V_0$, by \eqref{eq:approximate-chain-rule} and integration by parts we get
\begin{align*}
\int_{V_0}\xi w_nx\cdot\nabla z_n\,dx
&=
\int_{V_0}\xi x\cdot\nabla B_n(w_n)\,dx\,=\,
-\int_{V_0}
B_n(w_n)\operatorname{div}(\xi x)\,dx\\
&=
-N\int_{V_0}\xi B_n(w_n)\,dx
-
\int_{V_0}B_n(w_n)x\cdot\nabla\xi\,dx.
\end{align*}
Substituting this identity into
\eqref{eq:approximate-mixed-hessian}, we obtain
\begin{eqnarray}\label{eq:approximate-mixed-hessian-two}
-\int_{V_0}
\nabla w_n\cdot\nabla(\xi\,x\cdot\nabla u_n)\,dx
= 2\int_{V_0}\xi w_nz_n\,dx - N\int_{V_0}\xi B_n(w_n)\,dx 
-\int_{V_0}B_n(w_n)x\cdot\nabla\xi\,dx\nonumber\\
\qquad\quad
+
2\int_{V_0}w_n\nabla u_n\cdot\nabla\xi\,dx
+
2\int_{V_0}w_n
\sum_{i,j=1}^N (u_n)_{x_ix_j}x_j\xi_{x_i}\,dx
+
\int_{V_0}w_n(x\cdot\nabla u_n)\Delta\xi\,dx.\quad
\end{eqnarray}
Due to the uniform convergence $w_n\to w$ in bounded sets of $V$ and \eqref{zn to laplace u}, we also obtain $w_nz_n\rightarrow w\Delta u=|\Delta u|^q$ in $L^1(V_0)$. Moreover, using that $B_n(t)\rightarrow |t|^{q^\prime}/q$, locally uniformly, it follows that
\[
B_n(w_n)
\rightarrow
\frac1q|w|^{q'}
=
\frac1q|\Delta u|^q
\]
uniformly on $\overline{V_0}$.
Hence, the first three terms on the right-hand side of
\eqref{eq:approximate-mixed-hessian-two} converge to
\begin{align*}
&2\int_{V_0}\xi|\Delta u|^q\,dx
-
\frac Nq\int_{V_0}\xi|\Delta u|^q\,dx
-
\frac1q\int_{V_0}|\Delta u|^q
x\cdot\nabla\xi\,dx\\
&\qquad=
\frac{2q-N}{q}
\int_{V_0}\xi|\Delta u|^q\,dx
-
\frac1q\int_{V_0}|\Delta u|^q
x\cdot\nabla\xi\,dx.
\end{align*}
Finally, using again that
$u_n\rightarrow u$ in $W^{2,m}(V_0)$, and 
$w_n\rightarrow w$ in $W^{1,m}(V_0)$,
where $m>m'$, since $\xi$ and all its derivatives are smooth functions with compact
support in $V_0$, Hölder's inequality gives convergence of every
remaining term in
\eqref{eq:approximate-mixed-hessian-two}. In particular,
\[
\int_{V_0}
\nabla w_n\cdot\nabla(\xi\,x\cdot\nabla u_n)\,dx
\rightarrow
\int_{V_0}
\nabla w\cdot\nabla(\xi\,x\cdot\nabla u)\,dx.
\]
Passing to the limit in
\eqref{eq:approximate-mixed-hessian-two} proves
\eqref{eq:mixed-hessian-flux}.
\end{proof}

\begin{proposition}[Pohozaev identity]
\label{prop:pohozaev-identity}
Assume
$1<p<N,$ $ 1<q<N/2$, $0\leq\sigma<N$, $r>1$ and
suppose that $u\in \mathcal W\cap L^r_\sigma(\mathbb R^N)$ is a weak solution of 
\begin{equation*}
\Delta_q^2u-\Delta_pu
=
\lambda\dfrac{|u|^{r-2}u}{|x|^\sigma}+\tau_1|u|^{p^{*}-2}u
+
\tau_2|u|^{q^{**}-2}u
\quad\text{in }\R^N,
\end{equation*}
such that 
\begin{equation}
\label{eq:annular-regularity-pohozaev}
u\in W^{2,m}_{loc}(\mathbb R^N\backslash\{0\})\quad\text{and}\quad
w:=|\Delta u|^{q-2}\Delta u\in W^{1,m}_{loc}(\mathbb R^N\backslash\{0\}),
\end{equation}
for some $m>N$. Then
\begin{align}
&\frac{N-2q}{q}\int_{\mathbb R^N}|\Delta u|^q\,dx
+\frac{N-p}{p}\int_{\mathbb R^N}|\nabla u|^p\,dx
\nonumber\\&\ \ \ \ \ =\lambda\frac{N-\sigma}{r}
\int_{\mathbb R^N}\frac{|u|^{r}}{|x|^\sigma}\,dx
+\tau_1\frac N{p^*}
\int_{\mathbb R^N}|u|^{p^*}\,dx
+\tau_2\frac N{q^{**}}\int_{\mathbb R^N}|u|^{q^{**}}\,dx.
\label{eq:general-pohozaev-identity}
\end{align}
\end{proposition}

\begin{proof}
By Proposition \ref{prop:system-equivalence}, the pair \((u,w)\) satisfies
\begin{equation}\label{system fr}
\begin{cases}
\Delta w-\operatorname{div}(|\nabla u|^{p-2}\nabla u)=f_r(x,u),\\
\Delta u=|w|^{q'-2}w,
\end{cases}
\end{equation}
where
\[
f_r(x,u)=
\lambda\frac{|u|^{r-2}u}{|x|^\sigma}
+\tau_1|u|^{p^*-2}u
+\tau_2|u|^{q^{**}-2}u.
\]
 Choose \(\phi,\psi\in C^\infty([0,\infty))\) such that
\[
0\leq\phi,\psi\leq1,\qquad
\phi=0\text{ on }[0,1],\quad\phi=1\text{ on }[2,\infty),
\]
\[
\psi=1\text{ on }[0,1],\qquad
\psi=0\text{ on }[2,\infty),
\]
and define
\[
\xi_{\varepsilon,R}(x)=
\phi\left(\frac{|x|}{\varepsilon}\right)
\psi\left(\frac{|x|}{R}\right).
\]
Its derivatives are supported in $A_\varepsilon=\{\varepsilon<|x|<2\varepsilon\},$ $A_R=\{R<|x|<2R\},$ and
\begin{equation}
\label{eq:pohozaev-cutoff-bounds}
|x||\nabla\xi_{\varepsilon,R}|\leq C,\qquad
|x|^2|\Delta\xi_{\varepsilon,R}|\leq C.
\end{equation}
For a smooth open set $\Omega$ such that $\rm{supp}(\xi_{\varepsilon,R})\Subset \Omega\Subset\mathbb{R}^N\backslash\{0\}$, due to \eqref{eq:annular-regularity-pohozaev}, we can consider $u_n\in C^\infty(\Omega)$ such that $u_n\to u$ in $W^{2,m}(\Omega)$. Using $\phi_{n} := (\xi_{\varepsilon,R} x\cdot\nabla u_n)\in C^\infty_c(\Omega)$ as a test function in the first equation of \eqref{system fr} and then passing the limit as $n\to\infty$ we obtain
\begin{align}
&-\int_\Omega \nabla w\cdot\nabla
(\xi_{\varepsilon,R}x\cdot\nabla u) dx
+\int_\Omega |\nabla u|^{p-2}\nabla u\cdot
\nabla(\xi_{\varepsilon,R}x\cdot\nabla u) dx
\nonumber\\
&\qquad
=\int_\Omega f_r(x,u)\xi_{\varepsilon,R}x\cdot\nabla u dx.
\label{eq:tested-pohozaev-system}
\end{align}
From 
Lemma \ref{lem:mixed-hessian-flux} we know that
\begin{align*}
&-\int_\Omega \nabla w\cdot\nabla
(\xi_{\varepsilon,R}x\cdot\nabla u) dx
=\frac{2q-N}{q}\int_\Omega \xi_{\varepsilon,R}|\Delta u|^q dx
+E_{\varepsilon,R}^q,
\end{align*}
where
\begin{align*}
E_{\varepsilon,R}^q
={}&-\frac1q\int_\Omega |\Delta u|^q x\cdot\nabla\xi_{\varepsilon,R} dx
+2\int_\Omega  w\nabla u\cdot\nabla\xi_{\varepsilon,R} dx\\
&+2\int_\Omega  w\sum_{i,j=1}^Nu_{x_ix_j}x_j
(\xi_{\varepsilon,R})_{x_i} dx
+\int_\Omega  w(x\cdot\nabla u)\Delta\xi_{\varepsilon,R} dx.
\end{align*}

\noindent For the \(p\)-Laplacian term, the chain rule gives
\begin{align*}
&\int_\Omega  |\nabla u|^{p-2}\nabla u\cdot
\nabla(\xi_{\varepsilon,R}x\cdot\nabla u) dx
=\frac{p-N}{p}\int_\Omega \xi_{\varepsilon,R}|\nabla u|^p dx
+E_{\varepsilon,R}^p,
\end{align*}
where
\[
E_{\varepsilon,R}^p
=\int_\Omega |\nabla u|^{p-2}(x\cdot\nabla u)
\nabla u\cdot\nabla\xi_{\varepsilon,R} dx
-\frac1p\int_\Omega |\nabla u|^p x\cdot\nabla\xi_{\varepsilon,R}dx.
\]

\noindent The nonlinear terms satisfy
\begin{align*}
&\lambda\int_\Omega \frac{|u|^{r-2}u}{|x|^\sigma}
\xi_{\varepsilon,R}x\cdot\nabla udx
\nonumber
=-\lambda\frac{N-\sigma}{r}
\int_\Omega \xi_{\varepsilon,R}\frac{|u|^{r}}{|x|^\sigma}dx
-\frac{\lambda}{r}\int_\Omega \frac{|u|^{r}}{|x|^\sigma}
x\cdot\nabla\xi_{\varepsilon,R}dx,
\end{align*}
and, for every \(s>1\),
\begin{align*}
\int_\Omega |u|^{s-2}u\,\xi_{\varepsilon,R}x\cdot\nabla u dx
=-\frac Ns\int_\Omega \xi_{\varepsilon,R}|u|^sdx
-\frac1s\int_\Omega |u|^sx\cdot\nabla\xi_{\varepsilon,R}dx.
\end{align*}

\noindent It remains to remove the cutoffs. The Calderón--Zygmund estimate and
the first-order Hardy inequality applied to the components of
\(\nabla u\) give
\begin{equation*}
D^2u\in L^q(\mathbb R^N),\qquad
\frac{\nabla u}{|x|}\in L^q(\mathbb R^N)\quad\text{and}\quad\|D^2u\|_q+
\left\|\frac{\nabla u}{|x|}\right\|_q
\leq C\|\Delta u\|_q.
\end{equation*}
Also, $w\in L^{q'}(\mathbb R^N)$ and 
$\|w\|_{q'}^{q'}=\|\Delta u\|_q^q.$
Put \(\mathcal A_{\varepsilon,R}=A_\varepsilon\cup A_R\).
Using \eqref{eq:pohozaev-cutoff-bounds} and Hölder's inequality,
\begin{align*}
|E_{\varepsilon,R}^q|
\leq{}&
C\int_{\mathcal A_{\varepsilon,R}}|\Delta u|^qdx
+C\|w\|_{L^{q'}(\mathcal A_{\varepsilon,R})}
\left(
\|D^2u\|_{L^q(\mathcal A_{\varepsilon,R})}
+\left\|\frac{\nabla u}{|x|}\right\|_
{L^q(\mathcal A_{\varepsilon,R})}
\right),
\end{align*}
while
\[
|E_{\varepsilon,R}^p|
\leq C\int_{\mathcal A_{\varepsilon,R}}|\nabla u|^pdx.
\]
The nonlinear cutoff errors are bounded by
\[
C\int_{\mathcal A_{\varepsilon,R}}
\left(
\frac{|u|^{r}}{|x|^\sigma}
+\tau_1|u|^{p^*}
+\tau_2|u|^{q^{**}}
\right)dx.
\]
All these functions are globally integrable. Absolute continuity of
the integral therefore gives
\[
\lim_{R\to\infty}\lim_{\varepsilon\to0}
E_{\varepsilon,R}^q
=
\lim_{R\to\infty}\lim_{\varepsilon\to0}
E_{\varepsilon,R}^p=0.
\]
Similarly,
\[
\lim_{R\to\infty}\lim_{\varepsilon\to0}\int_\Omega \frac{|u|^{r}}{|x|^\sigma}
x\cdot\nabla\xi_{\varepsilon,R}dx=\lim_{R\to\infty}\lim_{\varepsilon\to0}\int_\Omega |u|^sx\cdot\nabla\xi_{\varepsilon,R}dx=0.
\]
Finally, since
\[
0\leq\xi_{\varepsilon,R}\leq1\quad\text{and}\quad
\lim_{R\to\infty}\lim_{\varepsilon\to0}
\xi_{\varepsilon,R}(x)=1
\quad(x\neq0),
\]
the dominated convergence theorem  ensures that \eqref{eq:general-pohozaev-identity} holds.
\end{proof}

As a standard result from the Pohozaev identity, we have the following nonexistence result.

\begin{corollary}
Let \(0\leq\sigma<2q\) and \(r>1\), and suppose that \(u\in \mathcal W\cap L^r_\sigma(\mathbb R^N)\) is a
weak solution of
\begin{equation}\label{pure power}
\Delta_q^2u-\Delta_pu
=
\frac{|u|^{r-2}u}{|x|^\sigma}
\qquad\text{in }\,\,\mathbb R^N
\end{equation}
satisfying \eqref{eq:annular-regularity-pohozaev}. If
$r\leq p^*_\sigma$ 
 or
$r\geq q^{**}_\sigma$,
then \(u\equiv0\). 
Moreover, for $\sigma$ and $r_\sigma$ as in \eqref{sigma_and_rsigma}, the reduced critical problem
\eqref{intro:problem-p-less-q-large}
has no nontrivial solution when \(\lambda\leq0\).
\end{corollary}

\begin{proof}
Testing the equation \eqref{pure power} with \(u\) and combining the resulting
identity with Proposition~\ref{prop:pohozaev-identity} gives
\[
\left(
\frac{N-2q}{q}-\frac{N-\sigma}{r}
\right)\int_{\mathbb R^N}|\Delta u|^q\,dx
+
\left(
\frac{N-p}{p}-\frac{N-\sigma}{r}
\right)\int_{\mathbb R^N}|\nabla u|^p\,dx
=0.
\]
Under either of the stated conditions on \(r\), the two coefficients
have the same sign, with at least one of them nonzero. Hence
\[
\int_{\mathbb R^N}|\Delta u|^q\,dx=\int_{\mathbb R^N}|\nabla u|^p\,dx=0,
\]
and therefore \(u\equiv0\).
For the reduced critical problem, 
the equation tested with \(u\) and the Pohozaev identity yield
\[
\lambda
\left(
\frac{N-2q}{q}
-\frac{N-\sigma}{r_\sigma}
\right)\int_{\mathbb R^N}
\frac{|u|^{r_\sigma}}{|x|^\sigma}\,dx
=
\left(
\frac{N-2q}{q}
-\frac{N-p}{p}
\right)\int_{\mathbb R^N}|\nabla u|^p\,dx.
\]
The two coefficients in parentheses are negative because
\(p<q^*\) and \(p<\sigma<2q\) imply
\[
\frac{N-2q}{q}
<
\frac{N-\sigma}{r_\sigma}
<
\frac{N-p}{p}.
\]
Thus, every nontrivial solution must satisfy \(\lambda>0\).
\end{proof}

\bigskip

\noindent\textbf{Acknowledgments:}
B.~Ribeiro acknowledges financial support from CNPq-Brazil through
grants 443594/2023-6, 314111/2023-9, and 201452/2024-3. E.~Gloss
acknowledges financial support from CNPq-Brazil through grants
201454/2024-6 and 443594/2023-6. K.~Perera  acknowledges financial support from CNPq-Brazil through
grant Universal 409764/2023-0. Part of this work was carried out while B.~Ribeiro and E.~Gloss were visiting the Florida Institute of
Technology. They gratefully acknowledge the institution for its hospitality and the stimulating research environment provided during
their visit.

\bigskip

\noindent\textbf{Data availability statement:}
This manuscript does not use any data.

\medskip

\noindent\textbf{Declaration of competing interests:}
The authors declare that they have no known competing financial interests or personal relationships that could have appeared to influence the work reported in this paper.

\medskip

\noindent\textbf{Declaration of generative AI and AI-assisted technologies
in the manuscript preparation process: } During the preparation of this work, the authors used ChatGPT (version 5.6 Sol)
(OpenAI) to assist with language editing and the organization of
parts of the exposition. The authors reviewed and edited all
resulting material and take full responsibility for the content of
the article.



\addcontentsline{toc}{section}{References}
\bibliographystyle{plain} 
\bibliography{Papers}

\end{document}